\documentclass[10pt]{article}
\usepackage{graphicx}
\usepackage{amsmath,amssymb}
\usepackage[left=2cm, right=2cm, top=2cm, bottom=2cm]{geometry}
\usepackage{color}
\usepackage[colorlinks=true]{hyperref}
\hypersetup{urlcolor=blue, citecolor=red}
\newcommand{\CQFD}{\hfill $\square$}
\newenvironment{proof}{{\bf Proof\\}}{\CQFD}

\newtheorem{thm}{\bf Theorem}[section]
\newtheorem{lem}[thm]{\bf Lemma}
\newtheorem{prop}[thm]{\bf Proposition}
\newtheorem{property}[thm]{\bf Properties}

\newtheorem{Def}{\bf Definition}[section]
\newtheorem{Hyp}{\bf Assumption}
\newtheorem{Hyps}[Hyp]{\bf Assumptions}
\newtheorem{rmk}{\bf Remark}[section]

\def\ds{\displaystyle}

\def\a{\alpha}
\def\b{\beta}
\def\t{\tilde}
\def\uu{\underline{u}}
\def\ou{\overline{u}}
\def\l{\lambda}
\def\phii{\varphi_i}

\def\O{\Omega}
\def\OT{\Omega\times(0,T)}
\def\OiT{\Omega_i\times(0,T)}

\def\Vv{\mathcal{V}}
\def\Pp{\mathcal{P}}
\def\dx{\delta x}
\def\dt{\delta t}

\def\eps{\varepsilon}
\newcommand{\R}{\mathbb{R}}
\newcommand{\N}{\mathbb{N}}
\newcommand{\Tt}{\mathcal{T}}

\newcommand{\Dd}{\mathcal{D}}

\newcommand{\Xx}{\mathcal{X}}
\newcommand{\llbrack}{\rm [ \! [}
\newcommand{\rrbrack}{\rm ] \! ]}

\def\ue{u^\epsilon}
\def\nn{\nonumber}
\begin{document}
%%-----------------------------
%%      the top matter
%%-----------------------------
\title{Finite volume scheme for two-phase flows in heterogeneous porous media involving capillary pressure discontinuities}

\author{Cl\'ement Canc\`es\footnote{ENS Cachan Bretagne, IRMAR, UEB, av. Robert Schuman, 35170 Bruz -- France, \href{mailto:clement.cances@bretagne.ens-cachan.fr}{\texttt{clement.cances@bretagne.ens-cachan.fr}}}
\footnote{The author is partially supported by the Groupement MoMaS}}
\date{}%
\maketitle
\begin{abstract} 
We study a one-dimensional model for two-phase flows in heterogeneous media, in which the capillary pressure functions can 
be discontinuous with respect to space. We first give a model, leading to a system of degenerated nonlinear parabolic equations spatially coupled  by nonlinear transmission conditions. 
We approximate the solution of our problem thanks to a monotonous finite volume scheme. The convergence of the underlying discrete solution 
to a weak solution when the discretization step tends to $0$ is then proven. We also show, under assumptions on the initial data, a uniform estimate on the flux, which is then used during the uniqueness proof. A density argument allows us to relax the assumptions on the initial data and to extend the existence-uniqueness frame to 
a family of solution obtained as limit of approximations. A numerical example is then given to illustrate the behavior of the model.
 \end{abstract}

\hskip 0.3cm {{\bf MSC subject classification.} 35R05, 65M12}

{\hskip 0.3cm {\bf keywords.} capillarity discontinuities, degenerate parabolic equation, finite volume scheme}
\maketitle
%%-----------------------------
%%      your text
%%-----------------------------
\section*{Introduction}

The models of immiscible two-phase flows in porous media are widely used in petroleum 
engineering in order to predict the positions where oil could be collected.  The discontinuities of the 
physical characteristics due to brutal change of lithology  play a crucial role in the phenomenon of oil trapping, 
preventing the light hydrocarbons from reaching the surface. It seems that the discontinuities with respect to the 
space variable of a particular function, called the capillary pressure, are responsible of the phenomenon of oil-trapping \cite{vDMdN95,BPvD03}.

In this paper, we consider one-dimensional two-phase flows in heterogeneous porous media, which are 
made of several homogeneous submedia. A simplified model of two-phase flow within this rock is described in the first section, leading to the definition of the weak solution.
The transmission conditions at the interface between the different submedia are written using 
the graph formalism introduced in \cite{CGP07} for the connection of the capillary pressures, 
which is simple to manipulate and allows to deal 
with any type of discontinuity of the domain, without any compatibility constraint, contrary to what occurs in \cite{NoDEA} 
and to a lesser extent in \cite{BPvD03,EEM06}.

The graph way to connect the capillary pressures at the interfaces is well suited to be discretized by a monotonous 
Finite Volume scheme. A discretization is proposed in the second  section of the paper. Adapting the material from 
the book of Eymard, Gallou\"et \& Herbin \cite{EGH00} to our case, it is shown that the discrete solution provided by the scheme converges, up to a subsequence, to a weak solution as the step of the discretization tends to $0$. 
The monotonicity of the transmission conditions is fundamental for proving the convergence of the scheme.

Unfortunately, we are not able to show the uniqueness of the weak solution to the problem, because of the lack of regularity.
As it will be shown in the fourth section, supposing that the fluxes are uniformly bounded with regard to space and time 
is sufficient to claim the uniqueness of the solution. The uniqueness proof is an adaptation of the one given in \cite{CGP07} to the case where the convection is not 
neglected. Here again, the monotonicity of the transmission conditions at the interfaces is 
strongly used.

The existence of a bounded flux solution is the topic of Section \ref{Bounded-fluxes}. It is shown that if the initial data is regular 
enough to ensure that the initial flux is bounded with respect to space, then the flux will remain bounded with respect to space and to time. Such a result has already been obtained in \cite{CGP07}, where a parabolic regularization 
of the problem had been introduced. A maximum principle on the flux follows. We also quote \cite{BPvD03}, in which a $BV$-estimate 
is shown on the flux. Since the monotonous schemes introduce some numerical diffusion, a strong analogy can be done between 
a uniformly parabolic regularization of the problem and the numerical approximation via a monotonous scheme. 
The convergence of the discrete solution to a bounded flux solution for regular enough initial data is thus naturally expected and stated in Theorem~\ref{thm_bounded}. The monotonicity of the transmission relations is essential during the proof.

We are able to prove the uniqueness of the bounded-flux solution to the problem using the doubling variable technique. This work performed in Section~\ref{Uniqueness_section} is summarized in Theorem~\ref{unicite_bounded}

In Section \ref{SOLA_section}, a density argument allows to extend the existence and uniqueness frame to any initial data, using the notion of SOLA (Solution Obtained as Limit of Approximation). It is a more restrictive notion than the notion of weak solution, even if we are not able to prove 
the existence of a weak solution which is not a SOLA.
The main result of the paper is given in Theorem~\ref{SOLA_thm}, which claims that the {\it whole sequence} of discrete solutions built using the finite volume scheme introduced in Section~\ref{scheme_part} converges towards the unique SOLA to the problem.

Finally, a numerical example is given in Section~\ref{numer_section}. This example gives an evidence of the entrapment of a certain quantity of oil under the interface.

 %%%%%%%%%%%%%%%%%%%%%%%%%%%%%%%%%%%
 %
 % MODELE
% 
%%%%%%%%%%%%%%%%%%%%%%%%%%%%%%%%%%%%
\section{Presentation of the problem}\label{model_part}
We consider a one-dimensional  heterogeneous porous medium, which is an apposition of homogeneous porous media, 
representing the different  geological layers. The physical properties of the medium 
only depend on the rock type and are piece-wise constant. 

%We restrict our study to the one-dimensional case, even if all the results stated in section~\ref{scheme_part} can be quite easily adapted to larger dimensions (see \cite{EEM06}).

For the sake of simplicity, we only deal with two geological layers of same size. A 
generalization to an arbitrary finite number of geological layers would only lead to notation difficulties.  
In the sequel, we denote by  
$\O=(-1,1)$ the heterogeneous porous medium, and by $\O_1=(-1,0)$, $\O_2=(0,1)$ the two homogeneous layers.
The interface between the layers is thus $\{ x=0\}$. $T$ is a positive real value.

We consider an incompressible and immiscible oil-water flow through $\O$. 
Writing the conservation of each phase, 
and using Darcy's law leads to: for all $(x,t)\in\OiT$, 
\begin{equation}\label{conserv_w}
\ds \phi_i\partial_t u - \partial_x\left[ \mu_{o,i}(u)\left(\partial_x P_{o,i} -\rho_o{ g}\right) \right]=0, 
\end{equation}
\begin{equation}\label{conserv_o}
\ds -\phi_i\partial_t u - \partial_x\left[ \mu_{w,i}(u)\left(\partial_x P_{w,i} -\rho_w{{g}}\right) \right]=0,
\end{equation}
where $\phi_i\in (0,1)$ is the porosity of the porous media $\O_i$, $u$ is the oil-saturation (then $(1-u)$ 
is the water-saturation), $\mu_{\beta,i}$ is the mobility of the phase $\beta=w,o$, where  $w$ stands for 
water, and $o$ for oil. We denote by $P_{\beta,i}$ the pressure of the phase $\beta$, by $\rho_\beta$ its density, and by ${g}$ the gravity. 

Adding (\ref{conserv_w}) and (\ref{conserv_o}) shows that :
$$
\partial_x q=0,
$$
where 
\begin{equation}\label{debit_total}
q=-\mu_{w,i}(u)\left(\partial_x P_{w,i} -\rho_w{g}\right) - \mu_{o,i}(u)\left(\partial_x P_{o,i} -\rho_o{{g}}\right)
\end{equation}
is the total flow-rate.
For the sake of simplicity, we suppose that $q$ does not depend on time, even if all the results presented below still 
hold for $q\in BV(0,T)$, as it is shown in \cite[chapter 4]{These}.

Using (\ref{debit_total}) 
in (\ref{conserv_w}) and (\ref{conserv_o}) yields:
\begin{equation}\label{eq_i}
\phi_i\partial_t u + \partial_x \left(\frac{\mu_{o,i}(u)}{\mu_{o,i}(u) + \mu_{w,i}(u)}q + \lambda_i(u) (\rho_o - \rho_w) g - 
\l_i(u) \partial_x(P_{o,i}-P_{w,i}) \right)=0,
\end{equation}
where 
$$\ds \l_i(u) =  \frac{\mu_{o,i}(u)\mu_{w,i}(u)}{\mu_{o,i}(u) + \mu_{w,i}(u)}.$$
One assumes that the capillary pressure $(P_{o,i}-P_{w,i})$ depends only on the saturation and of the rock type. More precisely, $(P_{o,i}-P_{w,i})=\pi_i(u)$, where $\pi_i(u)$ is supposed to be an increasing 
Lipschitz continuous function. The equation~\eqref{eq_i} becomes
\begin{equation}\label{eq_i_1}
\phi_i\partial_t u + \partial_x \left( f_i(u) -
\l_i(u)\partial_x\pi_i(u) \right)=0,
\end{equation} 
where 
$$
f_i(u) = \frac{\mu_{o,i}(u)}{\mu_{o,i}(u) + \mu_{w,i}(u)}q(t) + \l_i(u) (\rho_o - \rho_w) g.
$$
We do the following assumptions on the functions appearing in the equation.
\begin{Hyps}\label{mu_pi} For $i=1,2$, one has:
\begin{enumerate} 
\item $\pi_i$ is an increasing Lipschitz continuous function;
\item $\mu_{o,i}$ is an increasing Lipschitz continuous function on $[0,1]$, with $\mu_{o,i}(0)=0$;
\item $\mu_{w,i}$ is a decreasing Lipschitz continuous function on $[0,1]$, with $\mu_{w,i}(1)=0$.
\end{enumerate}
\end{Hyps}
\begin{rmk}
It is often supposed for such problems that the functions $\mu_{\beta,i}$ are monotonous in a large sense, and that there exist irreducible saturations
$s_{i},S_i \in (0,1)$, with $s_i <  S_i$, such that 
$$
\mu_{o,i}(u)=0  \textrm{ if } u \in [0,s_i], \quad \mu_{w,i}(u)=0  \textrm{ if } u \in [S_i,1].
$$ 
If we assume that the functions $\mu_{\beta,i}$ are strictly monotonous on their support, a convenient 
scaling would allow us to suppose that assumptions~\ref{mu_pi} are fulfilled.
\end{rmk}
We denote by $\phii(s)=\int_0^s \l_i(a) \pi_i'(a)da,$ then \eqref{eq_i_1} can be rewritten 
\begin{equation}\label{eq_i_2}
\phi_i\partial_t u + \partial_x \left( f_i(u) - 
\partial_x\phii(u) \right)=0.
\end{equation}
\begin{property}\label{prop_fonc} It follows directly from assumptions \ref{mu_pi} that for $i=1,2$ :
\begin{enumerate}
\item $f_i$ is Lipschitz continuous and $f_i(0) = 0$, $f_i(1) = q$;
\item $\l_i$ is Lipschitz continuous, and $\l_i(0)=\l_i(1) =0$, $\l_i(u)>0$ if $u>0$;
\item $\phii$ is an increasing Lipschitz continuous fulfilling $\phii(0) = 0$, $\phii'(0) = \phii'(1) = 0$.
\end{enumerate}
\end{property}
We deduce from the properties \ref{prop_fonc} that (\ref{eq_i_2}) is a degenerated nonlinear parabolic equation. 

Let us now focus on the transmission conditions through the interface $\{x=0\}$. 
We denote by $\a_i=\lim_{s\rightarrow0}\pi_i(s)$ and $\b_i=\lim_{s\rightarrow1}
\pi_i(s)$.
We define the monotonous graphs $\tilde\pi_i$ by:
\begin{equation}\label{tpii}
\t\pi_i(s)=\left\{\begin{array}{ll}
\ds \pi_i(s)  &\text{ if } s\in(0,1),\\
\ds (-\infty,\a_i]& \text{ if } s=0,\\
\ds [\b_i,+\infty) &\text{ if } s=1.
\end{array}\right.
\end{equation}

Let $u_i$ denote the trace of $u_{|\O_i}$ on $\{ x = 0\}$ (which is supposed to exist for the moment). The trace on $\{x=0\}$ from $\O_i$ of the pressure $P_{\beta,i}$ of the phase $\beta$  is still denoted by $P_{\beta,i}$.
As it is exposed in \cite{EEM06} (see also \cite{CGP07}), the pressure of the phase $\beta$ can be discontinuous through the interface $\{x=0\}$ in the case where it is missing in the upstream side. This can be written 
\begin{equation}\label{raccord_pressions_partielle}
\mu_{\beta,1}(u_1)(P_{\beta,1}-P_{\beta,2})^+-\mu_{\beta,2}(u_2)(P_{\beta,2}-P_{\beta,1})^+=0, \qquad \beta\in\{o,w\}.
\end{equation}

The conditions (\ref{raccord_pressions_partielle}) have direct consequences on the connection of the capillary pressures  through $\{x=0\}$. Indeed, if $0<u_1,u_2<1$, 
then the partial pressures $P_o$ and $P_w$ have both to be continuous, 
thus the connection of the capillary pressures $\pi_1(u_1)=\pi_2(u_2)$ is satisfied. If $u_1=0$ and $0<u_2 \le 1$, then $P_{o,1} \ge P_{o,2}$ and $P_{w,1} \le P_{w,2}$, thus $\pi_2(u_2)\le \pi_1(0)$. The same way, $u_1=1$ and $0 \le u_2<1$ 
implies $\pi_2(u_2)\ge \pi_1(1)$. 
Checking that the definition of the graphs $\t\pi_1$ and $\t\pi_2$ implies $\t\pi_1(0)\cap\t\pi_2(0)\neq\emptyset$, 
$\t\pi_1(1)\cap\t\pi_2(1)\neq\emptyset$, we can claim that (\ref{raccord_pressions_partielle}) implies: 
\begin{equation}\label{raccord_pi1}
\t\pi_1(u_1)\cap\t\pi_2(u_2)\neq\emptyset.
\end{equation}
The conservation of each phase leads to the connection of the fluxes on $\{x=0\}$. 
Denoting by $F_i$ the flux in $\O_i$, i.e. for all $x\in \O_i$, 
$$
F_i(x,t) = f_i(u)(x,t) - \partial_x \phii(u)(x,t),
$$
the connection of the fluxes through the interface can be written
\begin{equation}\label{raccord_flux}
F_1(0,\cdot) = F_2(0,\cdot),
\end{equation}
where \eqref{raccord_flux} has to be understood in a weak sense.

We now turn to the problem of the boundary conditions.
Because of technical difficulties occurring during section \ref{Uniqueness_section}, we want that the solution 
to the flow admits bounded fluxes, at least for regular initial data. This will force us to consider 
specific boundary conditions, which will involve bounded fluxes.\\
Let $G_i : (a,b) \mapsto G_i(a,b)$ $(i=1,2)$ be a function fulfilling the following properties :
\begin{itemize}
\item $G_i$ is Lipschitz continuous, non-decreasing w.r.t. its first argument, and non-increasing w.r.t. the second. 
\item for all $a\in[0,1]$, $G_i(a,a) = f_i(a)$. 
\end{itemize}
Let $\uu, \ou\in L^\infty(0,T)$, $0 \le \uu, \ou \le 1$ a.e., we choose the boundary condition 
\begin{equation}\label{cond_lim}
F_1(-1, t) = G_1(\uu(t), u(-1,t)), \qquad F_2(1,t) = G_2(u(1,t),\ou(t)).
\end{equation}
The way in which we approximate the boundary condition shall be judiciously compared with the discretization of the boundary conditions for scalar hyperbolic conservation laws using monotonous Finite Volume schemes (see \cite{Vov02}).

We consider an initial data $u_0\in L^\infty(\O)$, with $0 \le u_0 \le 1$, then we can write the initial-boundary-value problem:
\begin{equation}\label{P}\tag{$\Pp$}
\left\{
\begin{array}{l l}
\ds \phi_i \partial_t u +\partial_x\left[f_i(u)-\partial_x\phii(u)\right]=0 & \text{in } \OiT,\!\!\!\!\! \\
F_{1}(0,\cdot)=F_2(0,\cdot)  &  \text{on }(0,T),\\
\t\pi_1(u_1)\cap\t\pi_2(u_2)\neq\emptyset &  \text{on }(0,T),\\
u(t=0)=u_0 &\text{in }\O,\\
F_1(-1, t) = G_1(\uu(t), u(-1,t)) &  \text{on }(0,T), 
\\ F_2(1,t) = G_2(u(1,t),\ou(t)) &  \text{on }(0,T).
\end{array}
\right.
\end{equation}

We now define the notion of weak-solution
\begin{Def}\label{weak_sol_def}
A function $u$ is said to be a {\bf weak solution} to the problem~(\ref{P}) if it fulfills:
\begin{enumerate}
\item $u\in L^\infty (\OT)$, with $0\le u \le 1$;
\item for $i=1,2$, $\phii(u)\in L^2(0,T;H^1(\O_i))$;
\item for a.e. $t\in (0,T)$, $\t\pi_1(u_1(t))\cap\t\pi_2(u_2(t))\neq\emptyset$, where $u_i$ denotes the trace 
of $u_{|\O_i}$ on $\{x=0\}$;
\item for all $\psi\in \Dd(\overline\O\times[0,T[)$, denoting by $u(1,\cdot)$ and $u(-1,\cdot)$ the traces of $u$ on the boundary,
\begin{eqnarray}
&\ds \int_0^T\sum_{i=1,2}\int_{\O_i}\phi_i u(x,t)\partial_t\psi(x,t)dxdt\nn
+\sum_{i=1,2}\int_{\O_i}\phi_i u_0(x)\psi(x,0)dx&\\
&\ds +\int_0^T\sum_{i=1,2}\int_{\O_i}\left[f_i(u)(x,t)-\partial_x\phii(u)(x,t)\right]\partial_x\psi(x,t)dxdt&\nn\\
&\ds +\int_0^T G_1(\uu(t), u(-1,t))\psi(-1,t)dt - \int_0^T G_2(u(1,t),\ou(t))\psi(1,t) dt=0.&\label {weak_formulation}
\end{eqnarray}
\end{enumerate}
\end{Def}

%%%%%%%%%%%%%%%%%%%%%%%%%%%%%
%
% Schema VF
%
%%%%%%%%%%%%%%%%%%%%%%%%%%%%%
\section{The finite volume scheme}\label{scheme_part}
In this section, we build an implicit finite volume scheme in order to approximate a solution of~(\ref{P}). 
We will adapt the convergence proofs stated in \cite{CH99,EGH00,EEM06}, which are based on monotonicity properties of the scheme. This will allow us to claim the convergence in $L^p(\OT)$, up to a subsequence, of the discrete solutions built using the finite volume scheme towards a weak solution to the problem as step of the the discretization tends to $0$.

\subsection{The finite volume approximation}
We first need to discretize all the data, so that we can define an approximate problem through the finite volume 
scheme.

{\bf{Discretization of $\O$:}} for the sake of simplicity, we will only deal with uniform 
spatial discretizations. Let $N\in\N^\star$, one defines:
$$\left\{\begin{array}{l l}
\ds x_{j}={j}/{N},& \forall j\in\llbrack -N,N \rrbrack, \\
\ds x_{j+1/2}=\frac{j+1/2}{N},& \forall j\in\llbrack -N,N-1 \rrbrack. \\
\end{array}\right.$$
 One denotes by $\delta x=1/N$.

{\bf{Discretization of $(0,T)$:}} once again, we will only deal with uniform discretizations. 
 Let $M\in\N^\star$, one defines: for all $n\in\llbrack 0, M \rrbrack$, $t^n=nT/M$. One denotes by
$\delta t=T/M$. We denote by $\Dd$ the discretization of $\OT$ deduced of those of $\O$ and $(0,T)$.
 
 {\bf{Discretization of $u_0$:}}  $\forall j  \in\llbrack -N,N-1 \rrbrack$,
\begin{equation}\label{u0_d}
u_{0,\Dd}(x_{j+1/2})= u_{j+1/2}^0 = \frac{1}{\dx}\int_{x_j}^{x_{j+1}} u_0(x)dx.
\end{equation} 

 {\bf{Discretization of the boundary conditions:}}
 $\forall n\in \llbrack 0, M \rrbrack$, 
 $$
 \uu^{n+1} = \frac1\dt\int_{t^n}^{t^{n+1}} \uu(t) dt, \qquad  \ou^{n+1} = \frac{1}\dt\int_{t^n}^{t^{n+1}} \ou(t) dt.
 $$

  {\bf{The Finite Volume scheme:}}
 the first equation of~(\ref{P}) can be rewritten:
 $$
 \phi_i \partial_t u+ \partial_x F_{i}(x,t)=0, \qquad \textrm{ in } \OT
 $$
 with $F_{i}(x,t)=f_{i}(u)-\partial_x\phii(u)$. We consider the following implicit scheme: 
$\forall j\in\llbrack -N,N-1 \rrbrack$, $\forall n\in \llbrack 0, M-1 \rrbrack$,
\begin{equation}\label{schema_1}
\phi_i \frac{u_{j+1/2}^{n+1}-u_{j+1/2}^{n}}{\delta t}\delta x + F_{j+1}^{n+1}-F_j^{n+1}=0
\end{equation}
where $F_j^{n+1}$ is an approximation of the mean flux through ${x_j}$ on $(t^n, t^{n+1})$, and 
$i$ is chosen such that $(x_{j},x_{j+1})\subset \O_i$. This notation will hold all along the paper.
We choose a monotonous discretization of the flux: 
$\forall j\in\llbrack -N+1,-1\rrbrack \cup\llbrack 1,N-1 \rrbrack$, $\forall n\in \llbrack 0, M-1 \rrbrack$,
\begin{equation}\label{Flux_int}
F_j^{n+1}=G_{i}(u_{j-1/2}^{n+1}, u_{j+1/2}^{n+1})-\frac{\phii(u_{j+1/2}^{n+1})-\phii(u_{j-1/2}^{n+1})}{\dx},
\end{equation}
where $G_{i}$ is the same function as the one defined in \eqref{cond_lim}. We also define 
\begin{equation}\label{boundary_d}
F_{-N}^{n+1} = G_{1}(\uu^{n+1}, u_{-N+1/2}), \qquad  F_{N}^{n+1} = G_{2}(u_{N-1/2},\ou^{n+1}),
\end{equation}
\begin{eqnarray}
F_{0}^{n+1} & =  & G_{1}(u_{-1/2}^{n+1},u_{0,1}^{n+1})
-\frac{2(\varphi_1(u_{0,1}^{n+1})-\varphi_1(u_{-1/2}^{n+1}))}{\dx} \label{trans_flux_d}\\
& = & G_2(u_{0,2}^{n+1}, u_{1/2}^{n+1})-\frac{2(\varphi_2(u_{1/2}^{n+1})-\varphi_2(u_{0,2}^{n+1}))}{\dx},
\label{trans_flux_d2}
\end{eqnarray}
where $u_{0,1}^{n+1},u_{0,2}^{n+1}$ moreover satisfy 
\begin{equation}\label{raccord_pression_d}
\t\pi_1(u_{0,1}^{n+1})\cap\t\pi_2(u_{0,2}^{n+1})\neq\emptyset.
\end{equation}
\begin{rmk}\label{FN_borne}
The choice of the boundary conditions $F_{\pm N}^{n+1}$ has been done in order to ensure 
$$
\left| F_{-N}^{n+1} \right| \le \| G_1 \|_\infty < \infty, \qquad  
 \left| F_{N}^{n+1} \right| \le \| G_2 \|_\infty < \infty.
$$
\end{rmk}
Thanks to the following lemma, such a couple $(u_{0,1}^{n+1},u_{0,2}^{n+1})$ is unique in $[0,1]^2$, thus 
the discrete transmission conditions system \eqref{trans_flux_d}-\eqref{trans_flux_d2}-\eqref{raccord_pression_d} 
is well posed.
\begin{lem}\label{lem_trans_d}
For all $(a,b)\in [0,1]^2$, there exists a unique 
couple $(c,d)\in [0,1]^2$ such that:
\begin{equation}\label{trans_d}
\left\{\begin{array}{l}
\ds G_1(a,c)
-\frac{2(\varphi_1(c)-\varphi_1(a))}{\dx}= 
G_2(d,b)-\frac{2(\varphi_2(b)-\varphi_2(d))}{\dx},\\
\ds \t\pi_1(c)\cap\t\pi_2(d)\neq\emptyset.\\
\end{array}\right.
\end{equation}
Furthermore, $(a,b)\mapsto c$ and $(a,b)\mapsto d$ are continuous and nondecreasing w.r.t. each one of their arguments.
\end{lem}
\begin{proof}
For $i=1,2$,  $\t\pi_i^{-1}$ are continuous non-decreasing functions, increasing on $[\pi_i(0),\pi_i(1)]$  and constant otherwise.
Then we can build the continuous non-decreasing function $\Lambda$, defined by
$$
\Lambda :\left\{\begin{array}{r c l}
\R& \rightarrow & \R\\
 p& \mapsto & \ds G_2(\t\pi_2^{-1}(p), b)- G_1(a, \t\pi_1^{-1}(p))+\frac{2}{\dx}\left(\varphi_1\circ\t\pi_1^{-1}(p) - \varphi_1(a)+ 
\varphi_2\circ\t\pi_2^{-1}(p) -\varphi_2(b)\right).
\end{array}\right.
$$
For all $p$ such that $\Lambda(p)=0$, the couple $(\t\pi_1^{-1}(p), \t\pi_2^{-1}(p))$ is a solution to the discrete transmission conditions system \eqref{trans_flux_d}-\eqref{trans_flux_d2}-\eqref{raccord_pression_d}.
It is easy to check, using the monotonicity of the functions $G_i$ that for all $p\le \min \pi_i(0)$, $\Lambda(p) \le 0$. Symmetrically, for all $p \ge \max \pi_i(1)$, $\Lambda(p)\ge 0$. Thus there exists $p_\star$ such that $\Lambda(p_\star)=0$.

Suppose that there exists $i$ such that $p_\star \in (\pi_i(0),\pi_i(1))$, then since $\phii$ is increasing, $\Lambda$ is increasing 
on a neighborhood of $p_\star$, and then the solution to the system \eqref{trans_d} is unique.

Suppose now that $p_\star\notin\bigcup_i (\pi_i(0),\pi_i(1))$. Either $p_\star\le \min_i \pi_i(0)$, then $c=d=0$, or 
$p_\star \ge \max_i \pi_i(1)$, then $c=d=1$, or $p_\star \in [\pi_k(1), \pi_l(0)]$ for $k\neq l$. We can suppose without any 
loss of generality that $p_\star \in [\pi_1(1), \pi_2(0)]$, then the unique solution to the system \eqref{trans_d} is given by 
$c=1$, $d=0$.

To conclude the proof of the lemma, it only remains to check that $(a,b) \mapsto \Lambda$ is decreasing w.r.t. each one of its arguments, then the monotonicity 
of $\Lambda$ and $\t\pi_i^{-1}$ ensures that $(a,b)\mapsto c$ and $(a,b)\mapsto d$ are non-decreasing. 
\end{proof}

\subsection{Existence and uniqueness of the discrete solution} 

We will now work on the implicit finite volume scheme given 
by~(\ref{u0_d})-(\ref{raccord_pression_d}) to show that this approximate problem is well-posed.
\begin{Def}\label{disc_sol}
Let $N,M$ be two positive integers and $\Dd$ be the associated  discretization of $\OT$. One defines:
$$
\Xx_{\Dd,i}=\left\{\begin{array}{c} z\in L^\infty(\OiT)\ /\ \forall (x_j,x_{j+1})\subset \O_i, \forall n\in \llbrack 0,M-1 \rrbrack, \\
  z_{|{(x_j,x_{j+1})\times(t^n,t^{n+1}]}} \text{ is a constant}  \end{array} \right\},
$$
and
$$
\Xx_{\Dd}=\left\{ z\in L^\infty(\OT)\ /\  \forall i=1,2, \  z_{|\OiT}\in\Xx_{\Dd,i} \right\}.
$$
One defines
$u_\Dd(x,t)\in\Xx_\Dd$, called {\bf discrete solution}, given almost everywhere in $(-1,1)\times(0,T)$ by: 
for all $j\in\llbrack -N,N-1\rrbrack$, for all $n\in \llbrack 0, M-1\rrbrack$, 
$$ 
\left\{\begin{array}{l}
u_\Dd(x,0)=u_{0,\Dd}(x) = u_{j+1/2}^0 \text{ if } (x,t)\in(x_j,x_{j+1}), \\
u_\Dd(x,t)=u_{j+1/2}^{n+1}  \text{ if } (x,t)\in(x_j,x_{j+1})\times(t_n, t_{n+1}],
\end{array}\right.
$$
where $\left( u_{j+1/2}^{n+1} \right)_{j,n}$ are given by the scheme \eqref{schema_1}.
\end{Def}
The monotonicity of the flux $F_j^{n+1}$ w.r.t. $\left(u_{k+1/2}^{n+1}\right)_k$  allows us to rewrite 
the scheme \eqref{schema_1} under the form
\begin{equation}\label{comp_schema_0}
 H_{j+1/2} \left(u_{j+1/2}^{n+1}, u_{j+1/2}^{n} \left(u_{k+1/2}^{n+1}\right)_{k\neq j} \right) =0, 
\end{equation}
where $H_{j+1/2}$ is continuous, increasing w.r.t. its first argument, and non-increasing w.r.t. all the others.

\begin{Def}\label{super_sub_d}
A function $v_\Dd$ is said to be a {\bf discrete supersolution} (resp. $w_\Dd$ is a  discrete subsolution) if it belongs to $\Xx(\Dd)$, and 
if it satisfies: $\forall j\in \llbrack -N,N-1\rrbrack,$
\begin{eqnarray*}
&\ds H_{j+1/2} \left(v_{j+1/2}^{n+1}, v_{j+1/2}^{n} \left(v_{k+1/2}^{n+1}\right)_{k\neq j} \right)  \ge 0,
&  \\
 {\textrm{{\Huge(}resp.}} & H_{j+1/2} \left(w_{j+1/2}^{n+1}, w_{j+1/2}^{n} \left(w_{k+1/2}^{n+1}\right)_{k\neq j} \right)  \le 0
& \textrm{\Huge).}
\end{eqnarray*}
\end{Def}
\begin{rmk}\label{sup_sub_rmk}
A function $u_\Dd$ is a discrete solution to the scheme if and only if it is both a supersolution and a subsolution.
\end{rmk}
\begin{rmk}\label{01}
It follows from the definition of the scheme, particularly from the definitions of the discrete boundary conditions \eqref{boundary_d} and of the discrete fluxes at the interface 
\eqref{trans_flux_d}-\eqref{trans_flux_d2},  that the constant function equal to $1$ is a discrete supersolution, and 
that the constant function equal to $0$ is a discrete subsolution.
\end{rmk}
We now focus on the existence and the uniqueness of the discrete solution to the 
scheme. In order to prove the 
existence of a discrete solution, we first need an a priori estimate on it.
\begin{lem}\label{comp_d_lem}
Let $u_\Dd$ be a discrete solution to the scheme associated to the initial data $u_{0,\Dd}$,  
let $v_\Dd$ be a discrete supersolution associated to the initial data $v_{0,\Dd}$,  then for all 
$t\in [0,T]$, 
$$
\sum_{i=1,2}\int_{\O_i} \phi_i \left(u_\Dd(x,t) - v_\Dd(x,t)\right)^+ dx \le \sum_{i=1,2} \int_{\O_i} \phi_i \left(u_{0,\Dd}(x) - v_{0,\Dd}(x)\right)^+ dx.
$$
Symmetrically, if $w_\Dd$ is a subsolution associated to the initial $w_{0,\Dd}$, 
$$
\sum_{i=1,2}\int_{\O_i} \phi_i \left(u_\Dd(x,t) - w_\Dd(x,t)\right)^- dx \le \sum_{i=1,2} \int_{\O_i} \phi_i \left(u_{0,\Dd}(x) - w_{0,\Dd}(x)\right)^- dx.
$$
\end{lem}
\begin{proof}
Denoting by $a\top b = \max(a,b)$, and $a\bot b = \min(a,b)$, it follows from the monotonicity of the functions $H_{j+1/2}$ implies that 
$$
H_{j+1/2} \left(u_{j+1/2}^{n+1}, u_{j+1/2}^{n}\top w_{j+1/2}^n \left(u_{k+1/2}^{n+1}\top w_{k+1/2}^{n+1}\right)_{k\neq j} \right) \le 0,
$$
$$
H_{j+1/2} \left(w_{j+1/2}^{n+1}, u_{j+1/2}^{n}\top w_{j+1/2}^n 
\left(u_{k+1/2}^{n+1}\top w_{k+1/2}^{n+1}\right)_{k\neq j} \right) \le 0,
$$
where $w_\Dd$ is a subsolution.
Since $u_{j+1/2}^{n+1}\top w_{j+1/2}^{n+1}$ is either equal to $u_{j+1/2}^{n+1}$ or to $w_{j+1/2}^{n+1}$, 
\begin{equation}\label{comp_schema_1}
H_{j+1/2} \left(u_{j+1/2}^{n+1}\top w_{j+1/2}^{n+1}, u_{j+1/2}^{n}\top w_{j+1/2}^n 
\left(u_{k+1/2}^{n+1}\top w_{k+1/2}^{n+1}\right)_{k\neq j} \right) \le 0.
\end{equation}
Thanks to the conservativity of the scheme, subtracting \eqref{comp_schema_0} to 
\eqref{comp_schema_1}, and summing on $j\in \llbrack -N,N-1\rrbrack$ yields
 \begin{equation*}
 \sum_{i=1,2}\int_{\O_i} \phi_i \left(u_\Dd(x,t^{n+1}) - w_\Dd(x,t^{n+1})\right)^- dx \le  
 \sum_{i=1,2}\int_{\O_i} \phi_i \left(u_\Dd(x,t^{n}) - w_\Dd(x,t^{n})\right)^- dx.
 \end{equation*}
 Since this inequality holds for any $n\in \llbrack 0, M-1 \rrbrack$,
it directly gives: $\forall t\in [0,T]$, 
 \begin{equation}\label{comp_schema_2}
 \sum_{i=1,2}\int_{\O_i} \phi_i \left(u_\Dd(x,t) - w_\Dd(x,t)\right)^- dx \le  
 \sum_{i=1,2}\int_{\O_i} \phi_i \left(u_\Dd(x,0) - w_\Dd(x,0)\right)^- dx.
 \end{equation}
 The proof of the discrete comparison principle between a discrete solution and a discrete supersolution can be performed similarly.
\end{proof}

Let us now state the existence and the uniqueness of the discrete solution.
\begin{prop}\label{existence_unicite_d}
Let $u_0\in L^\infty(\O),$ $0\le u_0 \le 1$ a.e., then there exists a unique discrete solution $u_\Dd$ to the scheme, which furthermore fulfills 
$0 \le u_\Dd \le 1$ a.e.. Moreover, if $v_0$ stands for another initial data, $0 \le v_0 \le 1$, approximated by $v_{0 ,\Dd}$ following \eqref{u0_d}, and if we denote by $v_\Dd$ the corresponding discrete solution, then the following $L^1$-contraction principle holds;
$$
\sum_{i=1,2} \int_{\O_i} \phi_i \left( u_\Dd(x,t) - v_\Dd(x,t) \right)^\pm dx \le \sum_{i=1,2} \int_{\O_i} \phi_i \left( u_{0,\Dd}(x) - v_{0,\Dd}(x) \right)^\pm dx, \qquad \forall t\in [0,T].
$$ 
\end{prop}
\begin{proof}
It follows from Remark~\ref{01} and from Lemma~\ref{comp_d_lem} that the following $L^\infty$ a priori estimate holds:
$$
0 \le u_\Dd (x,t) \le 1, \qquad \textrm{ for all }t\in[0,T], \textrm{ for almost all } x\in\O.
$$
Thanks to this estimate, mimicking the proof given in \cite{EGGH98}, we can claim the existence of a discrete solution $u_\Dd$.
Suppose that $u_\Dd$ and $v_\Dd$ are two solutions associated to the initial data 
$u_{0,\Dd}$ and $v_{0,\Dd}$. As it was stressed in the remark~\ref{sup_sub_rmk}, both $u_\Dd$ and $v_\Dd$ are both 
discrete sub- and supersolutions. Then, Lemma~\ref{comp_d_lem} ensures that the following $L^1$-contraction principle holds:
$$
\sum_{i=1,2} \int_{\O_i} \phi_i \left( u_\Dd(x,t) - v_\Dd(x,t) \right)^\pm dx \le \sum_{i=1,2} \int_{\O_i} \phi_i \left( u_{0,\Dd}(x) - v_{0,\Dd}(x) \right)^\pm dx, \qquad \forall t\in [0,T].
$$ 
The uniqueness of the discrete solution $u_\Dd$ corresponding to the initial data $u_0$ follows.
\end{proof}
\subsection{The $L^2((0,T);H^1(\O_i))$ estimates}

The current subsection is devoted to the proof of the discrete energy estimate stated in Proposition \ref{prop_L2H1_d}. 
Since the discrete solutions are only piecewise constant, we need to introduce discrete semi-norms, which are discrete 
analogues to the $L^2((0,T);H^1(\O_i))$ semi-norms.

\begin{Def}\label{def_L2H1_d}
Let $i=1,2$, one defines the discrete $L^2(0,T; H^1(\O_i))$ semi-norms $|\cdot |_{1,\Dd,i}$ on $\Xx_{\Dd,i}$ by:
$\forall z \in \Xx_{\Dd,i}$,
$$ 
|z|_{1,\Dd,i}^2=\sum_{n=0}^{M-1}\dt \sum_{j\in J_{\text{int}, i}} 
\dx\left( \frac{Êz(x_{j+1/2},t^{n+1})-Êz(x_{j-1/2},t^{n+1})ÊÊ}{\dx} \right)^2,
$$
where $J_{\text{int}, 1}=\llbrack -N+1,-1 \rrbrack$ and  $J_{\text{int}, 2}=\llbrack 1,N-1 \rrbrack$.
\end{Def}
\begin{prop}\label{prop_L2H1_d}
For $i=1,2$, one defines the Lipschitz continuous increasing functions 
$$\ds \xi_i: s\mapsto \int_0^s \sqrt{\l_i(a)}\pi_i'(a)da.$$
There exists $C>0$ only depending on $\pi_i,\phi_i,T,G_i$ such that:
$$
\sum_{i=1,2} |\xi_i(u_\Dd)|_{1,\Dd,i}^2 \le C.
$$
\end{prop}
This estimate is the discrete analogue to: 
$$
\ds \sum_{i=1,2}  \int_0^T \int_{\O_i} |  \partial_x \xi_i(u)(x,t) |^2  dxdt \le C.
$$
In order to prove Proposition~\ref{prop_L2H1_d}, we will need the following technical lemma. 
To understand this lemma, first suppose that the total flow rate $q$ is $0$. Then, roughly speaking, it  claims that,  in the case where the capillary pressure is discontinuous at the interface, the discrete flux is oriented from the high capillary pressure to the low capillary pressure. Suppose now that $q\neq 0$.
In order to respect the conservation of mass, some fluid will have to go through the interface, but we keep a control on the energy.
\begin{lem}\label{monotony_d_lem1}
Let $(a,b)\in [0,1]^2$, and let $(c,d)\in [0,1]^2 $ be the unique solution to the system~(\ref{trans_d}), 
as stated in Lemma~\ref{lem_trans_d}, then the following inequality holds:
$$ (\pi_1(c)-\pi_2(d))\left(G_1(a,c) + \frac{\varphi_1(a)-\varphi_1(c)}{\dx/2}\right) =
(\pi_1(c)-\pi_2(d))\left(G_2(d,b) + \frac{\varphi_2(d)-\varphi_1(b)}{\dx/2}\right)\ge -|q| | \pi_1(c) - \pi_2(d) | .$$
\end{lem}
\begin{proof}
In this proof, we suppose that $\pi_1(0) \ge \pi_2(0)$ and $\pi_1(1)\ge \pi_2(1)$, 
the other cases do not bring any other difficulties.
One has $\t\pi_1(c)\cap\t\pi_2(d)\neq\emptyset$, so there are three different cases:
\begin{itemize}
\item \underline{$\pi_1(c)=\pi_2(d)$:} in this case, one has directly:
$$ (\pi_1(c)-\pi_2(d))\left( G_1(a,c) + \frac{\varphi_1(a)-\varphi_1(c)}{\dx/2}\right) = 0. $$
\item  \underline{$\pi_2(d)<\pi_1(0)$:}  the relation $\t\pi_1(c)\cap\t\pi_2(d)\neq\emptyset$ 
ensures that $c=0$, thus it follows from the monotonicity of $\varphi_1$ and $G_1$ that $\varphi_1(a) \ge \varphi_1(0)=0$, and $G_1(a,0) \ge G_1(0,0)=f_1(0)=0$. This gives:
$$ (\pi_1(c)-\pi_2(d)) \left(G_1(a,c) + \frac{\varphi_1(a)-\varphi_1(c)}{\dx/2}\right) \ge 0.
$$
\item  \underline{$\pi_1(c)>\pi_2(1)$:} this implies $d=1$. From the monotonicity of $\varphi_2$ and $G_2$, we deduce that $\varphi_2(b) \le \varphi_2(1)$ and $G_2(1,b) \ge G_2(1,1)=q$. This yields 
$$ (\pi_1(c)-\pi_2(d))\left(G_2(d,b) + \frac{\varphi_2(d)-\varphi_1(b)}{\dx/2}\right) \ge q| \pi_1(c) - \pi_2(d) |.
$$
\end{itemize}
\end{proof}
\noindent
{\it Proof of Proposition \ref{prop_L2H1_d}.}
First check that the scheme~(\ref{schema_1}) can be rewritten 
$$
\frac{u_{j+1/2}^{n+1} - u_{j+1/2}^{n}}{\dt} \dx + \left( F_{j+1}^{n+1} - f_i(u_{j+1/2}^{n+1}) \right) - \left( F_{j}^{n+1} - f_i(u_{j+1/2}^{n+1}) \right) = 0.
$$
We multiply the previous equation by $\dt\pi_i(u_{j+1/2}^{n+1})$ and sum on $j=-N,N-1$. 
This leads to
\begin{equation}\label{L2H1_10}
A^{n+1} + B^{n+1} + C^{n+1} + D^{n+1} + E^{n+1}= 0,
\end{equation}
where
\begin{eqnarray*}
A^{n+1}&=&\ds \sum_{j=-N}^{N-1} \phi_i \pi_i( u_{j+1/2}^{n+1} ) \left( u_{j+1/2}^{n+1} - u_{j+1/2}^{n} \right)\dx \ ;\\ 
B^{n+1}&=& \ds \sum_{j\notin \{-N,0,N \} } \dt \left[\begin{array}{c}
\pi_i(u_{j-1/2}^{n+1}) \left( G_i(u_{j-1/2}^{n+1}, u_{j+1/2}^{n+1}) - G_i(u_{j-1/2}^{n+1}, u_{j-1/2}^{n+1}) \right) \\
- \pi_i(u_{j+1/2}^{n+1}) \left( G_i(u_{j-1/2}^{n+1}, u_{j+1/2}^{n+1}) - G_i(u_{j+1/2}^{n+1}, u_{j+1/2}^{n+1})  \right) 
\end{array}\right] ;\\
C^{n+1} &=& \dt \sum_{j\notin \{-N,0,N \}} \left( \pi_i(u_{j+1/2}^{n+1}) - \pi_i(u_{j-1/2}^{n+1}) \right) \frac{  \phii(u_{j+1/2}^{n+1}) - \phii(u_{j-1/2}^{n+1})   }{\dx}\  ; \\
D^{n+1}& = &  \dt F_0^{n+1} \left( \pi_1(u_{-1/2}^{n+1}) - \pi_2(u_{1/2}^{n+1}) \right) {\dx}
  - \dt   \pi_1(u_{-1/2}^{n+1}) f_1(u_{-1/2}^{n+1})     + \dt    \pi_2(u_{1/2}^{n+1}) f_2(u_{1/2}^{n+1})      \ ;\\
E^{n+1} & = & \dt \pi_1(u_{-N+1/2}^{n+1}) \left(G_1\left(\uu^{n+1},u_{-N+1/2}^{n+1}\right) - G_1\left(u_{-N+1/2}^{n+1},u_{-N+1/2}^{n+1} \right) \right) \\
&& +  \dt  \pi_2(u_{N-1/2}^{n+1}) \left( G_2\left(u_{N-1/2}^{n+1},\ou^{n+1}\right) - G_2\left(u_{N-1/2}^{n+1}, u_{N-1/2}^{n+1}\right)  \right)\ .
\end{eqnarray*}
Denoting by $L_G$ a Lipschitz constant of both $G_i$, 
\begin{equation}\label{L2H1_E}
E^{n+1} \ge - \dt L_G \left(\|\pi_1\|_\infty + \| \pi_2 \|_\infty \right) \ .
\end{equation}
One has
\begin{eqnarray*}
F_0^{n+1} \left(\pi_1(u_{-1/2}^{n+1}) - \pi_2(u_{1/2}^{n+1}) \right) =&  & 
F_0^{n+1} \left(\pi_1(u_{-1/2}^{n+1}) - \pi_1(u_{0,1}^{n+1}) \right) %\\
+\ds  F_0^{n+1} \left(\pi_1(u_{0,1}^{n+1}) - \pi_2(u_{0,2}^{n+1}) \right) \\
&+& F_0^{n+1} \left(\pi_2(u_{0,2}^{n+1}) - \pi_2(u_{1/2}^{n+1}) \right).
\end{eqnarray*}
It has been proven in Lemma \ref{monotony_d_lem1} that there exists $C_1$ depending only on $q$ and $\pi_i$ 
such that 
\begin{equation}\label{L2H1_01}
 F_0^{n+1} \left(\pi_1(u_{0,1}^{n+1}) - \pi_2(u_{0,2}^{n+1}) \right) \ge C_1.
\end{equation}
Using the definition of $F_0^{n+1}$, it is then easy to check that there exists $C_2$ only depending on $G_i$, $q$, $\pi_i$, 
\begin{eqnarray}
D^{n+1} & \ge & 
\dt C_2 +\dt  \left(\pi_1(u_{-1/2}^{n+1}) - \pi_1(u_{0,1}^{n+1}) \right) \frac{\varphi_1(u_{-1/2}^{n+1}) - \varphi_1(u_{0,1}^{n+1})}{\dx/2} \nn\\ 
&&+ \dt \left(\pi_2(u_{1/2}^{n+1}) - \pi_2(u_{0,2}^{n+1}) \right) \frac{\varphi_2(u_{1/2}^{n+1}) - \varphi_2(u_{0,2}^{n+1})}{\dx/2}. \nn
\end{eqnarray}
Since $\pi_i$ is a non-decreasing function, $\ds \mathcal{G}_i:s\mapsto \int_0^s \phi_i\pi_i(a)da$ is convex, then: $\forall n \in \llbrack 0,M-1 \rrbrack$,
\begin{equation}\label{L2H1_A}
A^{n+1} \ge  \sum_{j=-N}^{N-1}   \left(  \mathcal{G}_i(u_{j+1/2}^{n+1} ) -  \mathcal{G}_i(u_{j+1/2}^{n} )  \right) \dx.
\end{equation}
We denote by $\ds \Psi_i(s) = \int_0^s \pi_i(\tau) f_i'(\tau) d\tau$, then an integration by parts  leads to
$$
\Psi_i(b) - \Psi_i(a) = \pi_i(a) (G_i(a,b) - f_i(a)) -  \pi_i(b) \left( G_i(a,b) - f_i(b) \right)  - \int_a^b \pi_i'(s) (f_i(s)-G_i(a,b)) ds.
$$
Since $f_i(s) = G_i(s,s)$, it follows from the monotonicity of $G_i$ and $\pi_i$ that 
$$
 \int_a^b \pi_i'(s) (f_i(s)-G_i(a,b)) ds \ge 0.
$$
Thus 
\begin{eqnarray*}
B^{n+1} & \ge & \dt \sum_{j\notin\{-N,0,N\}} \left( \Psi_i(u_{j+1/2}^{n+1}) - \Psi_i(u_{j-1/2}^{n+1})  \right) \\
& \ge & \dt \left(\Psi_1(u_{-1/2}^{n+1}) - \Psi_1(u_{-N+1/2}^{n+1}) + \Psi_2(u_{N-1/2}^{n+1}) -  \Psi_2(u_{1/2}^{n+1})\right).
\end{eqnarray*}
So there exists $C_3$, only depending on $\pi_i$, $f_i$ such that
\begin{equation}\label{L2H1_B}
B^{n+1} \ge \dt C_3.
\end{equation}
Let $\ds \xi_i:s\mapsto \int_0^s\sqrt{\l_i(a)}\pi_i'(a)da$, Cauchy-Schwarz inequality yields: $\forall (a,b)\in [0,1]^2$, 
$$(\pi_i(a)-\pi_i(b))(\phii(a)-\phii(b)) \ge (\xi_i(a)-\xi_i(b))^2.$$
This ensures that 
\begin{eqnarray}
C^{n+1} &\ge & \dt \sum_{j\notin\{-N,0,N\}} \frac{\left(\xi_i(u_{j+1/2}^{n+1}) -  \xi_i(u_{j-1/2}^{n+1})  \right)^2}{\dx} ;\label{L2H1_C} \\
D^{n+1} & \ge & \dt C_2 + \dt  \frac{\left(\xi_1(u_{0,1}^{n+1}) -  \xi_1(u_{-1/2}^{n+1})  \right)^2}{\dx/2} + 
 \dt  \frac{\left(\xi_2(u_{1/2}^{n+1}) -  \xi_2(u_{0,2}^{n+1})  \right)^2}{\dx/2}. \label{L2H1_D}
\end{eqnarray}
Summing \eqref{L2H1_10} on $n\in \llbrack 0, M-1 \rrbrack$, and taking into account \eqref{L2H1_E}, \eqref{L2H1_A}, \eqref{L2H1_B}, \eqref{L2H1_C}, \eqref{L2H1_D},  provides the 
existence of a quantity $C$, depending only on $T$, $\pi_i$, $G_i$, $\phi_i$ such that
\begin{equation}\label{L2H1_fin}
\sum_{i=1,2} \left| \xi_i (u_\Dd)  \right|^2_{1,\Dd,i} + \sum_{n=0}^{M-1} \dt \left(\frac{\left(\xi_1(u_{0,1}^{n+1}) -  \xi_1(u_{-1/2}^{n+1})  \right)^2}{\dx/2} +
 \frac{\left(\xi_2(u_{1/2}^{n+1}) -  \xi_2(u_{0,2}^{n+1})  \right)^2}{\dx/2}   \right) \le C.
\end{equation}
\hfill $\square$
\begin{rmk}\label{conv_trace_rmk} 
The estimate \eqref{L2H1_fin} is stronger than the one stated in Proposition \ref{prop_L2H1_d}, since it lets also appear some 
contributions coming from the interface. They will be useful in the sequel. Indeed, if we denote by $u_{\Dd,i}$ the trace of $\left(u_\Dd \right)_{|\O_i}$ on the 
interface $\{ x = 0\}$, and if we denote by $\gamma_{\Dd,i}(t) = u_{0,i}^{n+1}$ if $t\in (n\dt , (n+1)\dt]$, then it follows from \eqref{L2H1_fin} that
$$
\lim_{\dt,\dx\to 0} \left\| u_{\Dd,i} - \gamma_{\Dd,i}  \right\|_{L^p(0,T)} = 0, \qquad \forall p\in [1,\infty).
$$
Suppose that $u_{\Dd_i}$ converges in $L^p(0,T)$ towards a function $u_i$, as it will be proven later. Then, we directly obtain that $\gamma_{\Dd,i}$ also converges towards $u_i$. Moreover, for all $t>0$, $\t\pi_1(\gamma_{Dd,1}(t))\cap\t\pi_2(\gamma_{Dd,2}(t)) \neq \emptyset$. Since 
$$
F =  \{ (a,b)\in [0,1]^2\ | \ \t\pi_1(a) \cap \t\pi_2(b) \neq \emptyset \} \textrm{ is a closed set of } [0,1]^2,
$$
we can claim that $\t\pi_1(u_1) \cap \t \pi_2(u_2) \neq \emptyset$ a.e. in $(0,T)$. 
 \end{rmk}

\subsection{Compactness of a family of approximate solutions}\label{compactness} 
Let $(M_p)_{p\in \N},(N_p)_{p\in \N}$ be two sequences of positive integers tending to $+\infty$. 
We denote $\Dd_p$ the discretization of $\OT$ associated to $M_p$, and $N_p$.
The $L^\infty$-estimate stated in Proposition~\ref{existence_unicite_d} shows that there exists 
$u\in L^\infty(\OT)$, $0\le u \le 1$, such that, up to 
a subsequence, $u_{\Dd_p}\rightarrow u$ in the $L^\infty(\OT)$ weak-$\star$ sense as $p\rightarrow +\infty$.

We just need to prove that $u_{\Dd_p}\rightarrow u$ almost everywhere in $\OT$ 
to get the convergence of $(u_{\Dd_p})$ towards $u$ in 
$L^r(\OT)$ for any $1\le r < +\infty$. To apply Kolmogorov criterion (see e.g. \cite{Bre83}) 
we need some estimates on the space and time translates of $\xi_i(u_\Dd)$.
 \begin{lem}[space and time translates estimates]\label{translates}
 For all $\eta\in \R$, for $i=1,2$, one denotes $\O_{i,\eta}= \{x \in \O_i \ / \ (x+\eta) \in \O_i \}$, 
then the following estimate holds:  
  \begin{equation}\label{space_trans_1}
  \| \xi_i( u_\Dd)(\cdot+\eta,\cdot) -\xi_i( u_\Dd)(\cdot,\cdot) \|_{L^2(\O_{i,\eta}\times(0,T))} \le  
| \xi_i(u_\Dd) |_{1,\Dd,i} |\eta |(|\eta| + 2 \dx).   
  \end{equation}
One denotes $w_{i,\Dd}$ the function defined almost everywhere by:
 $$
 w_{i,\Dd}(x,t)=\left\{\begin{array}{lcl}
 \xi_i(u_\Dd)(x,t)& \text{ in }& \OiT, \\
 0 & \text{ in }& \R^2\backslash (\OiT).
 \end{array}
 \right.
 $$
 There exists $C_1$ depending only on $\pi_i,\phi_i,T,G_i$ and $C_2$ only depending on 
  $\pi_i,\phi_i,T, \l_i,G_i$ such that:
  \begin{equation}\label{space_trans_2}
  \forall \eta\in \R, \qquad \qquad \| w_{i,\Dd}(\cdot+\eta,\cdot) - 
  w_{i,\Dd}(\cdot,\cdot) \|_{L^2(\R^2)} \le  C_1 \eta,   
\end{equation}
\begin{equation}\label{time_trans}
 \forall \tau\in (0,T), \qquad \| w_{i,\Dd}(\cdot,\cdot+\tau) 
-w_{i,\Dd}(\cdot,\cdot) \|_{L^2(\O_{i}\times(0,T-\tau))} \le  C_2 \tau.   
  \end{equation}
 \end{lem} 
The previous lemma is in fact a compilation of Lemmata 4.2, 4.3 
and 4.6 of \cite{EGH00} adapted to our framework.
The estimates~(\ref{space_trans_2}) and (\ref{time_trans}) 
allows us to use the  Kolmogorov compactness criterion on
the sequence $(w_{i,\Dd_p})_{p\in \N}$, and thus, 
there exists $w_i\in L^2(\OiT)$ such that for almost every $(x,t)\in (\OiT)$, 
$\xi_i(u_{\Dd_p})(x,t) \rightarrow w_i(x,t)$, and then thanks to the 
$L^\infty$-estimate $0\le u_{\Dd_p}(x,t)\le 1$,
one can claim that $\xi_i(u_{\Dd_p}) \rightarrow w_i $ in $L^r(\OiT)$, for all 
$r\in [1,+\infty[$. Letting $p$ tend to $+\infty$ in 
(\ref{space_trans_1}) insures that $w_i$ belongs to $L^2(0,T;H^1(\O_i))$. 
Since $\xi_i^{-1}$ is a continuous function, we can identify the limit:
$$
w_i = \xi_i(u).
$$
Thus $\xi_i(u) \in L^2(0,T;H^1(\O_i))$, and since $\phii\circ\xi_i^{-1}$ is a Lipschitz function,  there exists 
$C$ depending only on $T,\pi_i,\phi_i, G_i,\lambda_i$ such that:
\begin{equation}\label{phii(u)}
\|\phii(u)\|_{L^2(0,T;H^1(\O_i))} \le C, 
\end{equation}
and that $\xi_i(u_{\Dd_p}) \rightarrow \xi_i(u)$, up to a subsequence, in 
$L^r(\OiT)$ as $p \rightarrow +\infty$ for any $r\in [1,+\infty)$. Since $\xi_i$, $i=1,2$, is an increasing function, one can claim that 
$u_{\Dd_p}$ converges a.e. in $\OT$ towards $u$, and then:
\begin{eqnarray}
u_{\Dd_p} \rightarrow u  &&\text{ in the } L^\infty(\OT)\text{-weak-}\star \text{ sense,}\label{result_conv_1}\\
u_{\Dd_p} \rightarrow u  &&\text{ in } L^r(\OiT), \ \forall r\in [1,+\infty[.\label{result_conv_2}
\end{eqnarray}

Roughly speaking, the approximation $u_\Dd$ obtained via a monotonous finite volume scheme, which introduces numerical diffusion,  is ``close" to the approximation $u^\epsilon$ obtained by adding additional diffusion $-\epsilon \Delta u^\epsilon$ to the problem. For such a continuous problem, 
we would have an estimate of type 
$$
\int_0^T \int_{\O_i} \left( \partial_x \xi_i(u^\epsilon) \right) dxdt \le C',
$$
which would lead to the relative compactness of the family $\left( \xi_i(\ue) \right)_{\epsilon>0}$ in $L^2(\OiT)$. Then the family 
$\left( \xi_i(\ue) \right)_{\epsilon>0}$ is also relatively compact in $L^2((0,T);H^s(\O_i))$ for all $s\in(1/2,1)$. This ensures that, up to a subsequence, 
the traces on the boundary and on the interface of $(\xi_i(\ue))$ converge in $L^2(0,T)$. The continuity of $\xi_i^{-1}$, and the $L^\infty$-estimate 
ensure that the traces of $\ue$ on the boundary and on the interface converge in $L^r(0,T)$, for all $r\in [1,\infty)$.

This sketch has to be modified in order to deal with discrete solutions, which do not belong to $L^2((0,T);H^s(\O_i))$ for $s>1/2$. Nevertheless, 
a convenient estimate on the translates at the boundary, based on the discrete $L^2((0,T);H^1(\O_i))$ estimate stated in Proposition \ref{prop_L2H1_d}, leads 
to the following convergence result, which is proven in the multidimensional case in \cite[Proposition 3.10]{Berlin}.
\begin{lem}\label{conv_trace}
Let $i=1,2$, and let $\alpha\in\partial\O_i$. We denote by  ${u}_{\alpha,\Dd_p,i}$ the trace of $(u_{\Dd_p})_{|\O_i}$ on $\{ x=\alpha \}$.
Then, one has: for all $r\in [1, \infty)$,
$${u}_{\alpha,\Dd_p,i} \rightarrow u_{|\O_i}(\alpha,\cdot) \text{ in } L^r(0,T) \text{ as }Êp\rightarrow +\infty.$$
\end{lem}
If we denote by $u_i(t) = u_{|\O_i}(0,t)$, it follows from the remark \ref{conv_trace_rmk} that 
$\t\pi_1(u_1) \cap \t\pi_2(u_2) \neq \emptyset$ a.e. in $(0,T).$ We can summarize all the results of this subsection in the following proposition :
\begin{prop}\label{compac_prop}
Let $\left(M_p\right)_p$, $\left(N_p\right)_p$ tend to $\infty$ as $p\to \infty$, and let $\left(\Dd_p\right)_p$ be the corresponding sequence of  discretizations 
of $\OT$. Let $\left(u_{\Dd_p}\right)_p$ be the sequence of corresponding discrete solutions to the scheme, then, up to a subsequence 
(still denoted by $\left( u_{\Dd_p} \right)_p$), there exists $u\in L^\infty(\OT)$, $0\le u \le 1$ a.e., with $\xi_i(u)\in L^2((0,T);H^1(\O_i))$ ($i=1,2$) such that:
$$
u_{\Dd_p} \to u \qquad \textrm{ a.e. in } \OT \qquad \textrm{ as } p\to \infty.
$$
Moreover, keeping the notations of Lemma \ref{conv_trace},
$$\begin{array}{rclccc}
 \ds u_{-1,\Dd_p,1}(t) &\rightarrow &u(-1,t) & \textrm{for a.e. } t\in (0,T) &\textrm{ as } p\to\infty, &\\
  \ds u_{1,\Dd_p,2}(t)& \rightarrow &u(1,t) & \textrm{for a.e. } t\in (0,T) &\textrm{ as } p\to\infty, &\\
  \ds  u_{0,\Dd_p,i}(t)& \rightarrow & u_{| \O_i} (0,t) = u_i(t) & \textrm{for a.e. } t\in (0,T) &\textrm{ as } p\to\infty, & i=1,2,\\
\end{array}
$$
and $\t\pi_1(u_1) \cap \t\pi_2(u_2) \neq \emptyset$ almost everywhere in $(0,T)$.
\end{prop}

%%%%%%%%%%%%%%%%%%%%
% CONVERGENCE
%%%%%%%%%%%%%%%%%%%%
\subsection{Convergence of the scheme}
We will now achieve the proof of the following result.
\begin{thm}\label{conv_weak_thm}
Let $(M_p)_{p\in \N},(N_p)_{p\in \N}$ be two sequences of positive integers tending to $+\infty$, and 
$(\Dd_p)_{p\in\N}$ the associated sequence of discretizations of $\OT$. Then,
up to a subsequence, the sequence $\left(u_{\Dd_p}\right)_p $ of the discrete solutions  converges in $L^r(\OT)$ for all $r\in [1,\infty)$ to a weak solution 
to the problem \eqref{P} in the sense of Definition \ref{weak_sol_def}.
\end{thm}
\begin{proof}
 As it has been seen in Proposition~\ref{compac_prop}, 
the discrete solution $u_{\Dd_p}$ converges, up to a subsequence, towards a function $u$ fulfilling all the regularity criteria 
to be a weak solution. In order to prove the convergence of the subsequence to a weak solution, it only remains to show that 
the weak formulation \eqref{weak_formulation} is satisfied by the limit $u$.

In order to simplify the proof of convergence of the scheme towards a weak solution, we will use a density result, which is a 
simple particular case of those stated in~\cite{Dro02}.
% lemme de densite
\begin{lem}\label{density}
Let $a,b\in \R$, $a<b$, then:
$\{ \psi\in C^\infty_c([a,b])/  \psi' \in C^\infty_c((a,b))\}$  is dense in  $W^{1,q}(a,b)$, $q\in [1,+\infty[$.
\end{lem}
This lemma particularly allows us, thanks to a straightforward generalization, to restrict the set of test functions $\psi$ for the weak formulation~(\ref{weak_formulation}) 
 to 
 $$\Tt = \{\psi\in \Dd(\overline{\O}\times [0,T)) /  \partial_x \psi \in \Dd((\cup_{i=1,2}\O_i)\times [0,T))\}.$$
Let $\psi\in \Tt$. 
For $j\in \llbrack -N_p, N_p-1 \rrbrack$, $n\in \llbrack 0,M_p-1 \rrbrack$, we denote 
by $ \psi_{j+1/2}^{n}=\psi(x_{j+1/2},t^n)$. Assume that $p$ is large enough to ensure:
\begin{equation}\label{val_psi_0}
\psi_{-1/2}^{n}=\psi_{1/2}^{n},\qquad \forall n\in \llbrack 0,M_p-1 \rrbrack,
\end{equation}
\begin{equation}\label{val_psi_N}
 \forall n\in \llbrack 0,M_p-1 \rrbrack, \qquad
 \left\{\begin{array}{l}
\psi_{-N+1/2}^{n}=\psi(-1,t^n),\\
 \psi_{N-1/2}^{n}=\psi(1,t^n).
 \end{array}\right.
\end{equation}
One has also
\begin{equation}\label{val_psi_T}
\psi_{j+1/2}^{M_p}=0, \qquad\forall j\in  \llbrack -N_p, N_p-1 \rrbrack.
\end{equation}
For $j\in \llbrack -N_p, N_p-1 \rrbrack$, $n\in \llbrack 0,M_p-1 \rrbrack$, let us multiply equation~(\ref{schema_1}) by 
$\psi_{j+1/2}^n\dt$, and sum on $j\in \llbrack -N_p, N_p-1 \rrbrack$, $n\in \llbrack 0,M_p-1 \rrbrack$, we get:
$$
\sum_{n=0}^{M_p-1}\sum_{j=-N_p}^{N_p-1}\phi_i(u_{j+1/2}^{n+1}-u_{j+1/2}^{n})\psi_{j+1/2}^n \dx + 
\sum_{n=0}^{M_p-1}\dt \sum_{j=-N_p}^{N_p-1} (F_{j+1}^{n+1}-F_{j}^{n+1})\psi_{j+1/2}^n=0,
$$
which can be rewritten thanks to~(\ref{val_psi_0}), \eqref{val_psi_N}, \eqref{val_psi_T}:
\begin{eqnarray}
&\ds \sum_{n=0}^{M_p-1} \sum_{j=-N_p}^{N_p-1}\phi_i u_{j+1/2}^{n+1} (\psi_{j+1/2}^n - \psi_{j+1/2}^{n+1}) \dx
- \sum_{j=-N_p}^{N_p-1} \phi_i  u_{j+1/2}^{0}\psi_{j+1/2}^0 \dx&\nn  \\
& + \ds \sum_{n=0}^{M_p-1}\dt \sum_{j\notin\{-N_p,0,N_p\}} F_j^{n+1}(\psi_{j-1/2}^{n} - \psi_{j+1/2}^{n} )& \nn \\
& - \ds \sum_{n=0}^{M_p-1}\dt F_{-N}^{n+1}\psi(-1,t^{n}) + \sum_{n=0}^{M_p-1}\dt F_N^{n+1} \psi(1,t^n)= 0.&\label{conv_10}
\end{eqnarray}
Using the definition of $F_j^{n+1}$, we obtain 
\begin{equation}\label{ABCDE_p} 
A_p + B_p + C_p + D_p + E_p = 0,
\end{equation}
with, using the notation $\psi_{-N-1/2}^n = \psi_{-N + 1/2}^n$, and $\psi_{N+1/2}^n = \psi_{N - 1/2}^n$,
\begin{eqnarray*}
A_p & = &  \sum_{n=0}^{M_p-1} \sum_{j=-N_p}^{N_p-1}\phi_i u_{j+1/2}^{n+1} (\psi_{j+1/2}^n - \psi_{j+1/2}^{n+1}) \dx \ ;\\
B_p & = & - \sum_{j=-N_p}^{N_p-1} \phi_i  u_{j+1/2}^{0}\psi_{j+1/2}^0 \dx \ ; \\
C_p & = & \sum_{n=0}^{M_p-1}\dt \sum_{j\notin\{-N_p,0,N_p\}} G_i(u_{j-1/2}^{n+1}, u_{j+1/2}^{n+1} )(\psi_{j-1/2}^{n} - \psi_{j+1/2}^{n} ) \ ; \\
D_p & = & - \sum_{n=0}^{M_p-1}\dt \sum_{j=-N_p}^{N_p-1} \phii(u_{j+1/2}^{n+1}) \frac{\psi_{j+3/2}^n - 2 \psi_{j+1/2}^n + \psi_{j-1/2}^n}{\dx}\ ; \\
E_p & = & - \ds \sum_{n=0}^{M_p-1}\dt G_1(\uu^{n+1},u_{-N+1/2}^{n+1})\psi(-1,t^{n}) + \sum_{n=0}^{M_p-1}\dt G_2(u_{N-1/2}^{n+1},\ou^{n+1}) \psi(1,t^n) \ .
\end{eqnarray*}
Since $h_{A_p}: (x,t)\mapsto \frac{\psi_{j+1/2}^n - \psi_{j+1/2}^{n+1}}\dt$ if $(x,t) \in (x_j,x_{j+1})\times (t^n,t^{n+1})$, converges uniformly towards $-\partial_t \psi$ as 
$p\to \infty$, and since $u_\Dd$ converges in $L^1(\OT)$ towards $u$,
\begin{equation}\label{A_p}
\lim_{p\to\infty} A_p = -\int_0^T \sum_{i=1,2} \int_{\O_i} \phi_i u(x,t) \partial_t \psi(x,t) dxdt.
\end{equation}
Thanks to the convergence in $L^1(\O)$ of $u_\Dd(x,0)$ towards $u_0$, we have
\begin{equation}\label{B_p}
\lim_{p\to\infty} B_p = - \sum_{i=1,2} \int_{\O_i} \phi_i u_0(x) \psi(x,0) dx.
\end{equation}
We have to rewrite $C_p = C^1_p + C^2_p$, with 
$$
C^1_p =  \sum_{n=0}^{M_p-1}\dt \sum_{j\notin\{-N_p,0,N_p\}} f_i(u_{j-1/2}^{n+1})(\psi_{j-1/2}^{n} - \psi_{j+1/2}^{n} ), 
$$
$$
C^2_p = \sum_{n=0}^{M_p-1}\dt \sum_{j\notin\{-N_p,0,N_p\}} \left( G_i( u_{j-1/2}^{n+1}, u_{j+1/2}^{n+1} ) - f_i(u_{j-1/2}^{n+1}) \right) (\psi_{j-1/2}^{n} - \psi_{j+1/2}^{n} ). 
$$
Thanks to Proposition~\ref{compac_prop}, the quantity $C_p^1$ converges towards $\ds - \int_0^T \sum_{i=1,2} \int_{\O_i} f_i(u)(x,t) \partial_x \psi(x,t) dxdt$ as $p\to \infty$. Concerning $C^2_p$, since 
$$
\left| G_i( u_{j-1/2}^{n+1}, u_{j+1/2}^{n+1} ) - f_i(u_{j-1/2}^{n+1} \right| =   \left| G_i( u_{j-1/2}^{n+1}, u_{j+1/2}^{n+1} ) -G_i( u_{j-1/2}^{n+1}, u_{j-1/2}^{n+1} ) \right| 
\le L_G | u_{j+1/2}^{n+1}- u_{j-1/2}^{n+1} |,
$$
and since $(x,t) \mapsto  \frac{\psi_{j-1/2}^{n} - \psi_{j+1/2}^{n} }\dx$ on $(x_{j-1/2},x_{j+1/2})\times(t^n,t^{n+1})$ is uniformly bounded by $\| \partial_x \psi \|_\infty$, 
$$
| C_p^2 | \le L_G \| \partial_x \psi \|_\infty \int_0^T \int_\O |\delta u_{\Dd_p}(x,t)| dxdt,
$$
where $\delta u_{\Dd_p}(x,t) = u_{j+1/2}^{n+1} - u_{j-1/2}^{n+1}$ on $(x_{j-1/2},x_{j+1/2})\times(t^n,t^{n+1})$, $(j\in\llbrack -N_p+1, N_p-1 \rrbrack)$, and $0$ otherwise. It is easy to check, thanks to the discrete $L^2((0,T);H^1(\O_i))$ estimates stated in Proposition~\ref{prop_L2H1_d}, that $\delta u_{\Dd_p}$ tends to $0$ in $L^1(\OT)$ as $p\to \infty$. Then
\begin{equation}\label{C_p}
\lim_{p\to\infty} C_p = - \int_0^T \sum_{i=1,2} \int_{\O_i} f_i(u)(x,t) \partial_x \psi(x,t) dxdt.
\end{equation}
Since, using Proposition~\ref{compac_prop}, $\phii(u_{\Dd_p})$ tends to $\phii(u)\in L^2((0,T);H^1(\O_i))$ in the $L^2(\OiT)$-topology, one has:
\begin{equation}\label{D_p}
\lim_{p\to\infty} D_p = - \int_0^T \sum_{i=1,2} \int_{\O_i} \phii(u)(x,t) \partial^2_{xx} \psi(x,t) dxdt = 
\int_0^T \sum_{i=1,2} \int_{\O_i}\partial_x \phii(u)(x,t) \partial_{x} \psi(x,t) dxdt.
\end{equation}
The strong convergence of the traces, stated in Proposition \ref{compac_prop} allows us to claim that 
\begin{equation}\label{E_p}
\lim_{p\to\infty} E_p = - \int_0^T   G_1(\uu(t), u(-1,t)) \psi(-1,t) dt  + \int_0^T G_2( u(1,t) , \ou(t) ) \psi(1,t) dt.
\end{equation}
We can thus take the limit for $p\to\infty$ in \eqref{ABCDE_p}, and it follows from \eqref{A_p}-\eqref{B_p}-\eqref{C_p}-\eqref{D_p}-\eqref{E_p} that 
$u$ fulfills the weak formulation \eqref{weak_formulation}.
\end{proof}
%%%%%%%%%%%%%%%%%%%%%%%%%%%%%%%%
%
%  UNIFORM BOUND ON THE FLUXES
%
%%%%%%%%%%%%%%%%%%%%%%%%%%%%%%%%
\section{Uniform bound on the fluxes}\label{Bounded-fluxes}
In this section, we show that, under some regularity assumptions on the initial data, there exists a solution with bounded fluxes. This 
existence result  is the consequence of some additional estimates on the discrete solution, and will be necessary to get the uniqueness 
result of Theorem~\ref{unicite_bounded}.
\begin{Def}\label{def_bounded}
A function $u$ is said to be a {\bf bounded-flux solution} to the problem~(\ref{P}) if:
\begin{enumerate}
\item $u$ is a weak solution to the problem~(\ref{P}) in the sense of Definition~\ref{weak_sol_def};
\item $\partial_x\phii(u)$ belongs to $L^\infty(\OiT)$.
\end{enumerate}
\end{Def}
In order to get an existence result, we need more regularity on the initial data, as stated below.
\begin{Hyps}\label{hyps_bounded}
We assume that:
\begin{enumerate}
\item $\partial_x\phii(u_0) \in L^{\infty}(\O_i),$ $0\le u_0 \le 1$;
\item $\t\pi_1(u_{0,1})\cap \t\pi_2(u_{0,2}) \neq \emptyset$, where $u_{0,i}$ is the trace of ${u_0}_{|\O_i}$ on $\{x=0\}$,
\end{enumerate}
\end{Hyps}
\begin{thm}\label{thm_bounded} 
Suppose that assumptions~\ref{hyps_bounded} 
are fulfilled. Let $(M_p)_{p\in\N}$, $(N_p)_{p\in\N}$ be two sequences of positive integers tending to $+\infty$.
Let $(u_{\Dd_p})_{p\in\N}$ be the sequence of the associated discrete solutions obtained via the finite volume 
scheme~(\ref{schema_1}), and let $u$ be an adherence value of the sequence 
$(u_{\Dd_p})_{p\in\N}$. Then $u$ is a bounded flux solution to the problem~(\ref{P}) in the sense of 
Definition~\ref{def_bounded}. This particularly ensures the existence of such a bounded-flux solution.
\end{thm}
All the section~\ref{Bounded-fluxes} will be devoted to the proof of Theorem~\ref{thm_bounded}.
We only need to verify the second point in Definition~\ref{def_bounded}, 
because we have already proven in Theorem~\ref{conv_weak_thm} that $u$ is a weak solution. So the aim of this section 
is to get the uniform bound on the fluxes. Such an estimate can be found in~\cite{CGP07} in the case where the convection is neglected.
It is obtained using a thin regular transition layer between $\O_1$ and $\O_2$, and a regularization of the initial data $u_0$. 
This technique was also used in~\cite{BPvD03} to get a $BV$-estimate on the fluxes in the case of a non-bounded 
domain $\O$, and for particular values of the data (which are supposed to be more regular). In this paper, we only deal with 
the discrete solution, which can be seen as a regularization of the solution to the continuous problem \eqref{P}. 

We extend the definitions of the discrete internal fluxes~(\ref{Flux_int})-(\ref{trans_flux_d2}) 
to the case $n=-1$, i.e. in the time $t=0$. 
For all $j\in \llbrack -N+1, N-1\rrbrack$, $j\neq 0$, 
\begin{equation}\label{Flux_int_0}
F_j^{0}=G_i(u_{j-1/2}^{0},u_{j+1/2}^{0})-\frac{\phii(u_{j+1/2}^{0})-\phii(u_{j-1/2}^{0})}{\dx}\ .
\end{equation}
Thanks to Lemma \ref{lem_trans_d}, there exists a unique couple $(u_{0,1}^0, u_{0,2}^0)$ solution to the system
\begin{eqnarray} 
 &&\ds F_{0}^{0}  =   G_1(u_{-1/2}^{0},u_{0,1}^0)
-\frac{\varphi_1(u_{0,1}^{0})-\varphi_1(u_{-1/2}^{0})}{\dx/2} \label{trans_flux_d_0}
 =  G_2(u_{0,2}^{0}, u_{1/2}^{0})-\frac{\varphi_2(u_{1/2}^{0})-\varphi_2(u_{0,2}^{0})}{\dx/2}, \label{flux_00_d}\\
 &&\ds \t\pi_1(u_{0,1}^{0})\cap\t\pi_2(u_{0,2}^{0})\neq\emptyset.\label{raccord_pression_d_0}
\label{trans_flux_d2_0}
\end{eqnarray}
\begin{rmk}
 $u_{0,1}^0$ and $u_{0,2}^0$ are given by Lemma~\ref{lem_trans_d}, and so they 
are different of $u_{0,1}$ and $u_{0,2}$. 
\end{rmk}
\begin{lem}\label{unif_u0}
There exists $C>0$ depending only on $u_0$, $\phii$, $q$ such that
$$
\max_{j\in\llbrack-N+1,N-1\rrbrack} |F_j^0| \le C.
$$
\end{lem}
\begin{proof}
Since $\phii(u_0)$ is a Lipschitz continuous function,  and $\phii^{-1}$ is continuous, ${u_0}_{|\O_i}$ is a continuous 
function, and there exists $y_{j+1/2} \in (x_j, x_{j+1})$ such that $u_{j+1/2}^0=u_0(y_{j+1/2})$. Then \eqref{Flux_int_0} 
can be rewritten 
$$
F_j^{0}=G_i(u_0(y_{j-1/2}),u_0(y_{j+1/2}))-\frac{\phii(u_0(y_{j+1/2}))-\phii(u_0(y_{j-1/2}))}{\dx}\ .
$$
Using the fact that  $\partial_x \phii(u_0) \in L^\infty(\O_i)$ gives directly: $\forall j\in \llbrack -N+1, N-1 \rrbrack\setminus\{0\}$, 
\begin{equation}\label{Fj0}
\left| F_j^0 \right|  \le \max_{i=1,2} \| G_i \|_\infty + 2\max_{i=1,2} \left( \|\partial_x \phii(u_0)\|_{L^\infty(\O_i)}\right).
\end{equation}
The monotony of the transmission conditions  $\t\pi_1(u_{0,1}^{0})\cap\t\pi_2(u_{0,2}^{0})\neq\emptyset$ and 
$\t\pi_1(u_{0,1})\cap\t\pi_2(u_{0,2})\neq\emptyset$ 
implies that either $u_{0,1}^{0} \ge u_{0,1}$ and $u_{0,2}^{0} \ge u_{0,2}$, or 
$u_{0,1}^{0} \le u_{0,1}$ and $u_{0,2}^{0} \le u_{0,2}$.
Assume for example that $u_{0,1}^{0} \ge u_{0,1}$ and $u_{0,2}^{0} \ge u_{0,2}$ ---the other case could be treated similarly--- 
then one deduce from~(\ref{flux_00_d}) that:
$$
 G_2(u_{0,2}, u_0(y_{1/2}))-\frac{\varphi_2(u_0(y_{1/2}))-\varphi_2(u_{0,2})}{\dx/2}   \le   F_0^0 \le G_1(u_0(y_{-1/2}),u_{0,1})
-\frac{\varphi_1(u_{0,1})-\varphi_1(u_0(y_{-1/2}))}{\dx/2},
$$
and so since $\phii(u_0)$ is a Lipschitz continuous function, 
$$
| F_{0}^{0} | \le \max_{i=1,2} \| G_i \|_\infty + 2\max_{i=1,2} \left( \|\partial_x \phii(u_0)\|_{L^\infty(\O_i)}\right).
$$
\end{proof}
%
% proposition clef
%
\begin{prop}\label{unif_flux}
There exists $C>0$ depending only on $u_0$, $\phii$, $G_i$,  such that
$$
\max_{j\in \llbrack -N+1,N-1\rrbrack}\left(\max_{n\in \llbrack 0,M\rrbrack} |F_j^n|\right) \le C.
$$
\end{prop}
\begin{proof}
For all $j\in \llbrack -N+1, N-1 \rrbrack\setminus\{0\}$, for all $n\in \llbrack 0,M-1 \rrbrack$, 
\begin{eqnarray}
F_j^{n+1}-F_j^n&=& \left( G_i(u_{j-1/2}^{n+1},u_{j+1/2}^{n+1})-G_i(u_{j-1/2}^n,u_{j+1/2}^{n+1})\right) + \left(  G_i(u_{j-1/2}^{n},u_{j+1/2}^{n+1}) - G_i(u_{j-1/2}^n,u_{j+1/2}^n)\right)  \nonumber \\
&&+ 
\left(\frac{\phii(u_{j-1/2}^{n+1})-\phii(u_{j-1/2}^{n})}{\dx}\right) 
- \left(\frac{\phii(u_{j+1/2}^{n+1})-\phii(u_{j+1/2}^{n})}{\dx}\right). \nonumber % \\
\end{eqnarray}
Thus, using (\ref{schema_1}) yields
\begin{eqnarray}
F_j^{n+1}-F_j^n&=& 
\frac{\dt}{\phi_i \dx}\frac{G_i(u_{j-1/2}^{n+1},u_{j+1/2}^{n+1})-G_i(u_{j-1/2}^n,u_{j+1/2}^{n+1})}{u_{j-1/2}^{n+1} - u_{j-1/2}^n} \left( F_j^{n+1} - F_{j-1}^{n+1} \right) \nn \\
&& +\frac{\dt}{\phi_i \dx} \frac{G_i(u_{j-1/2}^{n},u_{j+1/2}^{n+1}) - G_i(u_{j-1/2}^n,u_{j+1/2}^n)}{u_{j+1/2}^{n+1} - u_{j+1/2}^n} \left( F_{j+1}^{n+1} - F_j^n   \right) \nn \\
&& + \frac{\dt}{\phi_i \dx^2}\frac{\phii(u_{j-1/2}^{n+1})-\phii(u_{j-1/2}^{n})}{u_{j-1/2}^{n+1} - u_{j-1/2}^n} \left( F_j^{n+1} - F_{j-1}^{n+1} \right) \nn \\
&& - \frac{\dt}{\phi_i \dx^2}\frac{\phii(u_{j+1/2}^{n+1})-\phii(u_{j-1/2}^{n})}{u_{j+1/2}^{n+1} - u_{j+1/2}^n} \left( F_{j+1}^{n+1} - F_{j}^{n+1} \right) .\nn
\end{eqnarray}
The monotonicity of the scheme is once again crucial, since it implies that there exist two non-negative values $a_{j,j+1}^{n+1}, a_{j,j-1}^{n+1}$ such that 
\begin{equation}\label{Mmat_1}
\left( 1+ a_{j,j-1}^{n+1} + a_{j,j+1}^{n+1} \right) F_j^{n+1} - a_{j,j-1}^{n+1} F_{j-1}^{n+1} - a_{j,j+1}^{n+1} F_{j+1}^{n+1} =  F_j^n. 
\end{equation}
The monotonicity of the graph transmission condition \eqref{raccord_pression_d} ensures that either $u_{0,1}^{n+1} \ge u_{0,1}^n$ and 
$u_{0,2}^{n+1} \ge u_{0,2}^{n}$, or $u_{0,1}^{n+1} \le u_{0,1}^n$ and $u_{0,2}^{n+1} \le u_{0,2}^{n}$. Suppose for example that 
$u_{0,1}^{n+1} \ge u_{0,1}^n$ and $u_{0,2}^{n+1} \ge u_{0,2}^{n}$, the other case being completely symmetrical.
\begin{eqnarray}
F_0^{n+1}-F_0^n&=& \left( G_1(u_{-1/2}^{n+1},u_{0,1}^{n+1})-G_1(u_{-1/2}^n,u_{0,1}^{n})\right)  \nonumber \\
&&+ 
\left(\frac{\varphi_1(u_{-1/2}^{n+1})-\varphi_1(u_{-1/2}^{n})}{\dx/2}\right) 
- \left(\frac{\varphi_1(u_{0,1}^{n+1})-\varphi_1(u_{0,1}^{n})}{\dx/2}\right) \label{DF01}  \\ 
&=& \left( G_2(u_{0,2}^{n+1},u_{1/2}^{n+1})-G_2(u_{0,2}^n,u_{1/2}^{n})\right)  \nonumber \\
&&+ 
\left(\frac{\varphi_2(u_{0,2}^{n+1})-\varphi_2(u_{1/2}^{n})}{\dx/2}\right) 
- \left(\frac{\varphi_2(u_{0,2}^{n+1})-\varphi_2(u_{1/2}^{n})}{\dx/2}\right). \label{DF02}  
\end{eqnarray}
It follows from \eqref{DF01} and from the monotony of $G_1, \varphi_1$ that 
$$
F_0^{n+1}-F_0^n \le  \left( G_1(u_{-1/2}^{n+1},u_{0,1}^{n})-G_1(u_{-1/2}^n,u_{0,1}^{n})\right) + 
\left(\frac{\varphi_1(u_{-1/2}^{n+1})-\varphi_1(u_{-1/2}^{n})}{\dx/2}\right) .
$$
Similar computations as those done to obtain \eqref{Mmat_1} provide the existence of a non-negative value $a_{0,-1}^{n+1}$ such that 
\begin{equation}\label{Mmat_2}
\left( 1+ a_{0,-1}^{n+1}  \right) F_0^{n+1}  - a_{0,-1}^{n+1} F_{-1}^{n+1} \le F_0^n.
\end{equation}
Considering \eqref{DF02} instead of \eqref{DF01} shows the existence of a non-negative value $b_{0,1}^{n+1}$ such that 
\begin{equation}\label{Mmat_3}
\left( 1+ b_{0,1}^{n+1}  \right) F_0^{n+1}  - b_{0,1}^{n+1} F_{1}^{n+1} \ge F_0^n.
\end{equation}
We denote by $j_{\rm max}^{n+1}$ (resp. $j_{\rm min}^{n+1}$) the integer such that 
$$F_{j_{\rm max}^{n+1}}^{n+1} = \max_{j\in \llbrack -N,N\rrbrack} F_j^{n+1} \qquad
(\textrm{resp. }F_{j_{\rm min}^{n+1}}^{n+1} = \min_{j\in \llbrack -N,N\rrbrack} F_j^{n+1}). $$
Either $j_{\rm max}^{n+1} \in \{ -N, N\}$, then it follows from the remark \ref{FN_borne} that $\max_{j\in \llbrack -N,N\rrbrack} F_{j}^{n+1} \le \max_{i=1,2}\| G_i \|_\infty$, or $j_{\rm max}^{n+1} \in \llbrack -N+1, N-1 \rrbrack$. In the latter case, \eqref{Mmat_1} and \eqref{Mmat_2} imply 
$$
\max_j F_j^{n+1} = F_{j_{\rm max}^{n+1}}^{n+1} \le F_{j_{\rm max}^{n+1}}^{n} \le \max_j F_j^n.
$$
Similarly, \eqref{Mmat_1} and \eqref{Mmat_2} yield 
$$
\min_j F_j^{n+1} = F_{j_{\rm min}^{n+1}}^{n+1} \ge F_{j_{\rm min }^{n+1}}^{n} \ge \min_j F_j^n.
$$
We obtain a kind of discrete maximum principle on the discrete fluxes, which corresponds to the uniform bound on the continuous fluxes proven in \cite{CGP07}.
It follows from Lemma \ref{unif_u0} that
$$
\max_{n\in \llbrack 0, M \rrbrack} \left( \max_{j\in \llbrack -N+1, N-1 \rrbrack} \left| F_j^{n+1} \right| \right) \le   \max_{i=1,2} \| G_i \|_\infty + 2\max_{i=1,2} \left( \|\partial_x \phii(u_0)\|_{L^\infty(\O_i)}\right).
$$
\end{proof}
\noindent{\it Conclusion of proof of Theorem~\ref{thm_bounded} }
Let $(N_p)_{p\in\N}$, $(M_p)_{p\in\N}$ be two sequences of positive integers tending to $+\infty$, 
and let $(u_{\Dd_p})_{p\in\N}$ the sequences of associated discrete solutions. It has been seen in 
theorem~\ref{conv_weak_thm} that $(u_{\Dd_p})_p$ tends to a weak solution $u$ in $L^r(\OT)$, for all $r\in [1,+\infty)$.\\
Let $i=1,2$, let $(x,y)\in \O_i$, let $t\in(0,T]$. For $p$ large enough,there exists 
$j_0, j_1 \in J_\text{int}$ such that 
$x_{j_0} \le x \le x_{j_0+1}$ and $x_{j_1} \le y \le x_{j_1+1}$, and there exists $n$ such that $t \in (t^n, t^{n+1}]$.
\begin{eqnarray*}
\left|\phii(u_{\Dd_p})(x,t)-\phii(u_{\Dd_p})(y,t)\right| &=& \left| \phii(u_{j_0+1/2}^{n+1})-  \phii(u_{j_1+1/2}^{n+1}) \right| \\
& = & \left| \sum_{j=j_1+1}^{j_0} \phii(u_{j+1/2}^{n+1})-  \phii(u_{j-1/2}^{n+1}) \right|\\
& \le & \sum_{j=j_1+1}^{j_0} \left|  \phii(u_{j+1/2}^{n+1})-  \phii(u_{j-1/2}^{n+1}) \right|.
\end{eqnarray*}
Using the definition of the discrete flux~(\ref{Flux_int}):
 $$
 \left|\phii(u_{\Dd_p})(x,t)-\phii(u_{\Dd_p})(y,t)\right| \le   
\sum_{j=j_1+1}^{j_0} \dx  \left|  F_j^{n+1} - G_i(u_{j-1/2}^{n+1},u_{j+1/2}^{n+1})   \right|.
$$
We deduce from Proposition~\ref{unif_flux} that 
there exists $C>0$, depending only on $u_0$, $\phii$, $G_i$ such that:
$$
 \left|\phii(u_{\Dd_p})(x,t)-\phii(u_{\Dd_p})(y,t)\right| \le   \sum_{j=j_1+1}^{j_0} \dx C \le C( |x-y| +2\dx ).
$$
Letting $p$ tend towards $+\infty$, i.e. $\dx$ and $\dt$ towards $0$ gives
\begin{equation}\label{Lipschitz_1}
 \left|\phii(u)(x,t)-\phii(u)(y,t)\right| \le C|x-y|.
\end{equation}
So we deduce from (\ref{Lipschitz_1}) that $\partial_x \phii(u)\in L^\infty(\OiT)$.
\hfill$\square$
%%%%%%%%%%%%%%%%%%%%%%%%%%%%%%%
%
% UNICITE DE LA SOLUTION A FLUX BORNES
%
%%%%%%%%%%%%%%%%%%%%%%%%%%%%%%%
\section{Uniqueness of the bounded-flux solution}\label{Uniqueness_section}
This section is devoted to the proof of Theorem \ref{unicite_bounded}, which is an adaptation
of \cite[Theorem 5.1]{CGP07} to the case where the convection is taken into account.
\begin{thm}\label{unicite_bounded}
 If $u$,$v$ are bounded-flux solutions in the sense of Definition~\ref{def_bounded} associated 
to the initial data $u_0$, $v_0$,
then for all $p\in [1,+\infty[$, $u$ and $v$ belong to $C([0,T];L^p(\O))$, and  the following $L^1$-contraction principle holds:
$\forall t\in [0,T]$,
$$
\int_{\O_i}\phi_i(u(x,t)-v(x,t))^\pm dx\le \int_{\O_i}\phi_i(u_0(x)-v_0(x))^\pm dx.
$$
This particularly implies the uniqueness of the bounded flux solution to the problem~(\ref{P}) 
\end{thm}
Obtaining a $L^1$-contraction principle for a nonlinear parabolic equation is classical. We refer for example to \cite{AL83,GMT94,Otto96,Car99,MPT02,MV02,BP05} for the case of homogeneous domains, and for 
boundary conditions of Dirichlet or Neumann type. We have to adapt the proof of the $L^1$-contraction principle to our problem, and thus particularly to the boundary conditions and to the transmission conditions at the interface.

We need to introduce the \emph{cut-off} functions $\rho_{\a}^{\eps}\in C^{0,1}(\O,\R^+)$ defined by 
$$
\rho_{\a}^{\eps}(x) = \left(\frac{\eps - |x-\a|}{\eps}\right)^+ .
$$

\begin{lem}\label{lem_interf}
For all $\theta\in \Dd^+([0,T))$,
$$
\liminf_{\eps\to 0} \int_0^T \theta(t)\sum_{i=1,2} \int_{\O_i} \left( {\rm sign}_\pm(u-v) (f_i(u) - f_i(v)) - 
\partial_x \left(\phii(u) - \phii(v)  \right)^\pm \right) \partial_x \rho_0^\eps(x) dxdt \ge 0.
$$
\end{lem}
\begin{proof} 
We define the subsets  of $(0,T)$ 
$$
E_{u>v} = \Big\{ t\in (0,T) \left| \  u_1(t) > v_1(t) \textrm{ or } u_2(t) > v_2(t) \Big\} \right. ,
$$
$$
E_{u\le v} = \left( E_{u>v}\right)^c  = \Big\{ t\in (0,T) \left| \ u_1(t) \le v_1(t) \textrm{ and } u_2(t) \le v_2(t) \Big\} \right. .
$$
Since the trace on $\{x=0\}$ of the function ${\rm sign}_+(u-v)(f_i(u) -f_i(v))$ is equal to $0$ for all 
$t\in E_{u \le v}$, it is easy to check that
\begin{equation*}\label{il13}
\lim_{\eps\to 0}   \int_{E_{u \le  v}}  \theta \sum_{i=1,2} \int_{\O_i} {\rm sign}_+(u-v)(f_i(u)-f_i(v))\partial_x \rho_0^\eps dxdt = 0.
\end{equation*}
Thanks to the fact that the trace of $\left(\phii(u)-\phii(v)\right)^+$ is equal to $0$ on the interface, one has also,  
\begin{eqnarray*}
&\ds\liminf_{\eps\to 0}   \int_{E_{u \le v}}  \theta \sum_{i=1,2} \int_{\O_i} \partial_x \left( \phii(u) - \phii(v)  \right)^+\partial_x \rho_0^\eps dxdt \ge 0. &     \end{eqnarray*}
  This particularly implies that 
 \begin{equation}\label{il15}
 \liminf_{\eps\to 0} \int_{E_{u \le v}}  \theta \sum_{i=1,2} \int_{\O_i} \left[ {\rm sign}_+(u-v) (f_i(u)-f_i(v)) - \partial_x \left( \phii(u) - \phii(v)  \right)^+\right] \partial_x \rho_0^\eps dxdt \ge 0.
 \end{equation}

Since $u,v$ are two weak solutions, subtracting their corresponding weak formulation~\eqref{weak_formulation} for the test function $\psi(x,t) = \theta(t) \rho_0^\eps(x)$ leads to
\begin{eqnarray*}
&\ds \int_0^T \sum_{i=1,2} \int_{\O_i} \phi_i (u-v)  \rho_0^\eps \partial_t \theta dxdt + 
\sum_{i=1,2} \int_{\O_i} \phi_i (u_0-v_0)   \rho_0^\eps \theta(0) dx     &\\
&\ds    + \int_0^T \theta \sum_{i=1,2} \int_{\O_i} \left( (f_i(u)-f_i(v)) - \partial_x (\phii(u)-\phii(v)) \right) \partial_x \rho_0^\eps  dxdt =0.&
\end{eqnarray*}
Since $\rho_0^\eps$ tends to $0$ in $L^1(\O)$ as $\eps \to 0$, one has:
$$
\lim_{\eps\to 0} 
\ds  \int_0^T \sum_{i=1,2} \int_{\O_i} \phi_i (u-v)  \rho_0^\eps \partial_t \theta dxdt + 
\sum_{i=1,2} \int_{\O_i} \phi_i (u_0-v_0)    \rho_0^\eps \theta(0) dx 
 =0, 
$$
thus 
\begin{equation}\label{il05}
\lim_{\eps\to 0}  \int_0^T \theta \sum_{i=1,2} \int_{\O_i} \left( (f_i(u)-f_i(v)) - \partial_x (\phii(u)-\phii(v)) \right) \partial_x \rho_0^\eps dxdt =0.
\end{equation}
Thanks to the $L^\infty(\OT)$ bound on the fluxes, one has 
$$
\left| \int_0^T \theta \sum_{i=1,2} \int_{\O_i} \left( (f_i(u)-f_i(v)) - \partial_x (\phii(u)-\phii(v)) \right) \partial_x \rho_0^\eps dxdt \right| \le C \| \theta \|_{L^1(0,T)},
$$
then, using a density argument,  \eqref{il05} holds for all $\theta\in L^1(0,T)$.

Replacing $\theta$ by $\theta \chi_{E_{u>v}}$ in \eqref{il05}, and  splitting the positive and the negative parts $a = a^+ - a^-$, we obtain
\begin{eqnarray}
&\ds   \int_{E_{u>v}} \theta \sum_{i=1,2} \int_{\O_i} \left( {\rm sign}_+(u-v)(f_i(u)-f_i(v)) - \partial_x (\phii(u)-\phii(v))^+ \right) \partial_x \rho_0^\eps dxdt    & \nn\\
= &\ds    \int_{E_{u>v}}  \theta \sum_{i=1,2} \int_{\O_i} \left( {\rm sign}_-(u-v)(f_i(u)-f_i(v)) - \partial_x (\phii(u)-\phii(v))^- \right) \partial_x\rho_0^\eps dxdt     & + r(\eps),
\label{il10}
\end{eqnarray}
with $$\lim_{\eps\to 0} r(\eps) =0.$$
For almost every $t\in E_{u>v}$, it follows from the monotonicity of the graph relations for the capillary pressure 
$$
\t\pi_1(u_1) \cap \t\pi_2(u_2) \neq \emptyset, \qquad \t\pi_1(v_1) \cap \t\pi_2(v_2) \neq \emptyset,
$$
that $t$ belongs to $E_{u\ge v} =  \Big\{ t\in (0,T) \left| \ u_1(t) \ge v_1(t) \textrm{ and } u_2(t) \ge v_2(t) \Big\} \right.$. So we obtain exactly in the same way that for \eqref{il15}, that 
$$
\liminf_{\eps\to 0} \int_{E_{u > v}}  \theta \sum_{i=1,2} \int_{\O_i} \left[ {\rm sign}_-(u-v) (f_i(u)-f_i(v)) - \partial_x \left( \phii(u) - \phii(v)  \right)^-\right]  \partial_x \rho_0^\eps dxdt \ge 0.
$$
It follows directly from \eqref{il10} that 
\begin{equation}\label{il30}
\liminf_{\eps\to 0} \int_{E_{u > v}}  \theta \sum_{i=1,2} \int_{\O_i} \left[ {\rm sign}_+(u-v) (f_i(u)-f_i(v)) - \partial_x \left( \phii(u) - \phii(v)  \right)^+\right]  \partial_x \rho_0^\eps dxdt \ge 0.
\end{equation}
Adding \eqref{il15} and \eqref{il30} achieves the proof of Lemma \ref{lem_interf}.
\end{proof}

%\begin{lem}\label{lem_boundary}
%For all $\theta\in \Dd^+([0,T))$, 
%\begin{eqnarray}
%\liminf_{\eps\to 0}& \ds   \left[- \int_0^T \theta(t) {\rm sign}_\pm (u(-1,t)-v(-1,t))\left( G_1(\uu(t), u(-1,t))-G_1(\uu(t),v(-1,t)) \right)\right. dt &\nn \\
%&\ds \left. + \int_0^T \theta \int_{\O_1} \left( {\rm sign}_\pm (u-v) \left( f_1(u) - f_1(v) \right) - \partial_x ( \varphi_1(u) - \varphi_1(v) )^\pm \right) \partial_x \rho^\eps_{-1} dxdt \right]
% \ge 0. &\label{bl1} \\
% %
% \liminf_{\eps\to 0}& \ds   \left[ - \int_0^T \theta(t) {\rm sign}_\pm (u(1,t)-v(1,t))\left( G_2(u(1,t),\ou(t))-G_2(v(1,t),\ou(t)) \right) \right. dt &\nn \\
%&\ds \left. + \int_0^T \theta \int_{\O_2} \left( {\rm sign}_\pm (u-v) \left( f_2(u) - f_2(v) \right) - \partial_x ( \varphi_2(u) - \varphi_2(v) )^\pm \right) \partial_x \rho^\eps_{1} dxdt \right]
% \ge 0. &\label{bl2} 
%\end{eqnarray}
%\end{lem}
%\begin{proof}
%For obvious reasons of symmetry between the roles played by $x=-1$ and $x=1$, and the ones played by $u$ and $v$, we only have to prove that 
%\begin{eqnarray}
% \liminf_{\eps\to 0}& \ds   \left[ - \int_0^T \theta(t) {\rm sign}_\pm (u(1,t)-v(1,t))\left( G_2(u(1,t),\ou(t))-G_2(v(1,t),\ou(t)) \right) \right. dt &\nn \\
%&\ds \left. + \int_0^T \theta \int_{\O_2} \left( {\rm sign}_\pm (u-v) \left( f_2(u) - f_2(v) \right) - \partial_x ( \varphi_2(u) - \varphi_2(v) )^\pm \right) \partial_x \rho^\eps_{1} dxdt \right]
% \ge 0. &\label{bl20} 
%\end{eqnarray}
%\end{proof}
%
\begin{lem}\label{lem_boundary2}
For all $\theta\in \Dd^+([0,T))$, 
\begin{equation}\label{bl3}
\liminf_{\eps\to 0} \int_0^T \theta \int_{\O_1} \left( {\rm sign}_\pm (u-v) \left( f_1(u) - f_1(v) \right) - \partial_x ( \varphi_1(u) - \varphi_1(v) )^\pm \right) \partial_x \rho^\eps_{-1} dxdt \ge 0, 
\end{equation}
\begin{equation}\label{bl4}
\liminf_{\eps\to 0} \int_0^T \theta \int_{\O_2} \left( {\rm sign}_\pm (u-v) \left( f_2(u) - f_2(v) \right) - \partial_x ( \varphi_2(u) - \varphi_2(v) )^\pm \right) \partial_x \rho^\eps_{1} dxdt \ge 0.
\end{equation}
\end{lem}
\begin{proof}
For the sake of simplicity, we will only prove 
$$
\liminf_{\eps\to 0} \int_0^T \theta \int_{\O_2} \left( {\rm sign}_+ (u-v) \left( f_2(u) - f_2(v) \right) - \partial_x ( \varphi_2(u) - \varphi_2(v) )^+ \right) \partial_x \rho^\eps_{1} dxdt \ge 0,
$$
but all the steps of the proof can be extended to the other cases. We denote by $F_{u>v}$ and $F_{u\ge v}$ the subsets of $(0,T)$ given by
$$
F_{u>v} = \left\{ t\in (0,T)\ | \ u(1,t) > v(1,t) \right\},\qquad F_{u\le v} =\left( F_{u>v}\right)^c= \left\{ t\in (0,T)\ | \ u(1,t) \le v(1,t) \right\}.
$$
Let $\eps>0$.
For almost every $t\in F_{u\le v}$, one has 
$$
\int_{\O_2} \partial_x \left( \varphi_2(u)(x,t) - \varphi_2(v)(x,t) \right)^+ \partial_x \rho^\eps_{1}(x) dx \le 0.
$$
Then, using the fact that for almost every $t\in F_{u\le v}$, the trace of ${\rm sign}_+(u(\cdot,t) - v(\cdot,t))\left( f_2(u)(\cdot,t) - f_2(v)(\cdot,t)  \right)$  on $\{ x=1\}$ is equal to $0$,
\begin{equation}\label{bl25}
\liminf_{\eps\to 0}\int_{F_{u\le v}} \!\!\!\theta   \int_{\O_2}  \left( {\rm sign}_+(u - v)\left( f_2(u) - f_2(v)  \right) - \partial_x (\varphi_2(u)-\varphi_2(v) )^+
\right) \partial_x \rho_1^\eps dx \ge 0.
\end{equation}
%%
%% Convergence domin\'ee
We deduce from the weak formulation that for all $\theta\in \Dd([0,T))$, 
\begin{equation}\label{bl30}
\lim_{\eps\to 0} \int_0^T \theta \left( \int_{\O_2} \left(f_2(u) - \partial_x  \varphi_2(u)\right) \partial_x \rho^\eps_1 dx  -G_2(u(1,t),\ou(t) )\right) dt =0.
\end{equation}
Since the fluxes $f_2(u) - \partial_x \varphi_2(u)$ and $f_2(v) - \partial_x \varphi_2(v)$ belong to $L^\infty(\O_2\times(0,T))$, a density argument, which has already been used during the proof of Lemma \ref{lem_interf}, allows us to claim that 
\eqref{bl30} still holds for any $\theta\in L^1(0,T)$. So, it particularly holds if we replace $\theta$ by $\theta \chi_{F_{u>v}}$
This leads to 
$$
\lim_{\eps\to 0}  \int_{F_{u<v}} \theta \int_{\O_2} \Big(f_2(u)-f_2(v) - \partial_x (\varphi_2(u)-\varphi_2(v))\Big) \partial_x \rho^\eps_1 dxdt =   \int_{F_{u>v}} \theta(t) \left( G_2(u(1,t),\ou(t))-G_2(v(1,t),\ou(t)) \right) dt.
$$
It follows from the monotonicity of $G_2$ that
$$ \forall t\in F_{u>v},\qquad  
G_2(u(1,t),\ou(t)) \ge G_2(v(1,t),\ou(t)),
$$
thus 
\begin{equation}\label{bl40}
\liminf_{\eps\to 0}  \int_{F_{u<v}} \theta \int_{\O_2} \Big(f_2(u)-f_2(v) - \partial_x (\varphi_2(u)-\varphi_2(v))\Big) \partial_x \rho^\eps_1 dxdt \ge 0.
\end{equation}
In order to conclude the proof of Lemma \ref{lem_boundary2}, it only remains to check that 
\begin{eqnarray}
&\ds \liminf_{\eps\to 0}  \int_{F_{u<v}} \theta \int_{\O_2} \Big(f_2(u)-f_2(v) - \partial_x (\varphi_2(u)-\varphi_2(v))\Big) \partial_x \rho^\eps_1 dxdt &\nn\\
=&\ds \liminf_{\eps\to 0}  \int_{F_{u<v}} \theta \int_{\O_2} \Big({\rm sign}_+(u-v)(f_2(u)-f_2(v)) - \partial_x (\varphi_2(u)-\varphi_2(v))^+\Big) \partial_x \rho^\eps_1 dxdt . & \label{bl50}
\end{eqnarray}
Since $\varphi_2^{-1}$ is a continuous function, $u(\cdot,t)$ can be supposed to be continuous on $\O_2$ for almost every $t$ in $(0,T)$. Particularly, for almost every $t\in F_{u>v}$, there exists a neighborhood $\Vv_t$ of $\{x=1\}$ such that $u(\cdot,t)>v(\cdot,t)$ for all $x\in \Vv_t$. On $\Vv_t$, one has 
$$
R(x,t) = \Big(f_2(u)-f_2(v) - \partial_x (\varphi_2(u)-\varphi_2(v)) \Big)- \Big({\rm sign}_+(u-v)(f_2(u)-f_2(v)) - \partial_x (\varphi_2(u)-\varphi_2(v))^+\Big) = 0.
$$
Then, for almost every $t\in F_{u>v}$, 
$$
\lim_{\eps\to 0} \int_{\O_2} R(x,t) \partial_x \rho^\eps_1(x) dx =0.
$$
Moreover, since the fluxes $f_2(u) - \partial_x \varphi_2(u)$ and $f_2(v) - \partial_x \varphi_2(v)$ belong to 
$L^\infty(\O_2\times(0,T))$, there exists $C>0$ not depending on $\eps$ such that  for almost every $t$, 
$$
\left|  \int_{\O_2} R(x,t) \partial_x \rho^\eps_1(x) dx \right|\le C.
$$
We deduce from the dominated convergence theorem that 
$$
\lim_{\eps\to 0} \int_0^T \int_{\O_2} R(x,t) \partial_x \rho^\eps_1(x) dx = 0.
$$
This particularly implies that \eqref{bl50} holds. This achieves the proof of Lemma \ref{lem_boundary2}.
\end{proof}

\noindent{\it Proof of the Theorem \ref{unicite_bounded}.}
First, since $u$ and $v$ are weak solutions to a parabolic equation, they are also entropy solutions 
(see \cite{GMT94}, \cite{Car99}), and it has been proven in \cite{Cont_L1} that $u$ and $v$ belong to 
$C([0,T],L^p(\O))$, in the sense that there exists $\t u,\t v \in  C([0,T],L^p(\O))$ such that 
$u = \t u$, $v=\t v$ almost everywhere in $\OT$.

Let $u$ and $v$ be two weak solutions, then some classical computations, based on the doubling variable technique applied on both the time and the space variable (see e.g. \cite{GMT94}, \cite{Car99})  yields yields that 
for any $\psi\in\Dd^+(\O_i\times[0,T))$ 
\begin{eqnarray}
&\ds \int_0^T \int_{\O_i}\phi_i(u(x,t)-v(x,t))^\pm
\partial_t\psi(x,t)dxdt\nn 
+\ds \int_{\O_i}\phi_i (u_0(x)-v_0(x))^\pm \psi(x,0)dx &\nn\\
&+\ds  \int_0^T  \int_{\O_i}
{\rm sign}_{\pm}(u(x,t) - v(x,t))\left(f_i(u)(x,t)-f_i(v)(x,t)\right)
\partial_x\psi(x,t)dxdt&\nn \\
&-\ds   \int_0^T  \int_{\O_i} \partial_x(\phii(u)(x,t)-\phii(v)(x,t))^\pm \partial_x\psi(x,t)dxdt \ge 0.
\label{entro_i} &
\end{eqnarray}
Let $\theta\in \Dd^+([0,T))$, then summing \eqref{entro_i} with respect to $i=1,2$,
choosing $$\psi (x,t) = \theta(t) \left(1 - \rho^\eps_{-1}(x)- \rho^\eps_{0}(x)- \rho^\eps_{1}(x)\right)$$ 
as test function, and letting $\eps$ tend to $0$ leads to, thanks to Lemmata \ref{lem_interf} and \ref{lem_boundary2} : 
\begin{equation}\label{bT}
\int_0^T\partial_t\theta(t)\sum_{i=1,2} \int_{\O_i}\phi_i(u(x,t)-v(x,t))^\pm dxdt + \sum_{i=1,2} \int_{\O_i} \phi_i (u_0(x) - v_0(x))^\pm \theta(0)dx \ge 0.
\end{equation}
Since $u,v$ belong to $C([0,T];L^1(\O))$, the relation \eqref{bT} still holds for any $\theta\in BV(0,T)$ with $\theta(T^+) = 0$.
Let $t\in [0,T]$, we choose  $\theta = \chi_{[0,t)}$ in  \eqref{bT},  obtaining this way the $L^1$-contraction and comparison principle stated in the Theorem \ref{unicite_bounded}.
\hfill$\square$
%%%%%%%%%%%%%%%%%%%%%%%%%%%%%%%%%%%%%%
%
%  THE SOLA APPROACH
%
%%%%%%%%%%%%%%%%%%%%%%%%%%%%%%%%%%%%%%
\section{Solutions obtained as limit of approximations}\label{SOLA_section}

We aim in this section to extend the existence-uniqueness result obtained in Theorems~\ref{thm_bounded} 
and \ref{unicite_bounded} for any initial data $u_0\in L^\infty(\O)$, $0\le u_0 \le 1$ a.e.. We are unfortunately not able to prove the uniqueness of the weak solution to the problem~(\ref{P}) 
in such a general case, but we are able to prove the existence and the uniqueness of the solution obtained as limit of 
approximation by bounded flux solution. Moreover, this limit is the weak solution obtained via the convergence of the implicit 
scheme~(\ref{schema_1}) studied previously.
\begin{Def}
\label{SOLA_def} 
A function $u$ is said to be a {\bf SOLA} (solution obtained as limit of approximation) to the problem~(\ref{P}) if it fulfils:
\begin{itemize}
\item $u$ is a weak solution to the problem~(\ref{P}),
\item there exists a sequence $(u_\nu)_{\nu\in \N}$ of bounded flux solutions such that 
$$u_n \rightarrow u \text{ in } C([0,T];L^1(\O)),  \text{ as } n\rightarrow +\infty.$$
\end{itemize}
\end{Def}
\begin{thm}\label{SOLA_thm}
Let $u_0\in L^\infty(\O)$, $0\le u_0 \le 1$ a.e., then there exists a unique SOLA 
$u$ to the problem~(\ref{P}) in the sense of Definition~\ref{SOLA_def}. \\
Furthermore, if $(M_p)_{p\in\N}$, $(N_p)_{p\in\N}$ are to sequences of positive integers 
tending to $+\infty$, and if $(u_{\Dd_p})_{p\in\N}$ is the corresponding sequence of discrete solutions, then 
$u_{\Dd_p}\rightarrow u$ in $L^r(\OT)$, $r\in[1,+\infty)$.
\end{thm}
\begin{proof}
The set 
$$
\mathcal{E}= \Big\{ u_0\in L^\infty(\O) \ \Big| \ 0 \le u_0 \le 1, \ \partial_x \phii(u_0)\in L^\infty(\O_i), \ \t\pi_1(u_{0,1})\cap \t\pi_2(u_{0,2}) \neq \emptyset\Big\}
$$ 
is dense in $\{ u_0\in L^\infty(\O) \ | \ 0 \le u_0 \le 1\}$ for the $L^1(\O)$-topology.
Then we can build a sequence $\left( u_{0,\nu} \right)_{\nu\in\N}$ such that 
$$
\lim_{\nu\to\infty} \left\| u_{0,\nu} - u_0 \right\|_{L^1(\O)} = 0.
$$
Let $\left( u_\nu \right)_\nu$ be the corresponding sequence of bounded flux solutions, then we deduce from the Theorem~\ref{unicite_bounded} that for all $\nu,\mu\in \N$,
\begin{equation}\label{SOLA1}
\forall t\in [0,T],\qquad \sum_{i=1,2} \int_{\O_i} \phi_i \left( u_\nu (x,t) - u_\mu(x,t) \right)^\pm dx \le \sum_{i=1,2} \int_{\O_i} \phi_i \left( u_{0,\nu} (x) - u_{0,\mu}(x) \right)^\pm dx.
\end{equation}
Then $\left( u_\nu \right)_\nu$ is a Cauchy sequence in $C([0,T];L^1(\O))$, thus it converges towards $u \in C([0,T];L^1(\O))$, and
\begin{equation}\label{SOLA12}
\forall t\in [0,T],\qquad \sum_{i=1,2} \int_{\O_i} \phi_i \left( u_\nu (x,t) - u(x,t) \right)^\pm dx \le \sum_{i=1,2} \int_{\O_i} \phi_i \left( u_{0,\nu} (x) - u_{0}(x) \right)^\pm dx.
\end{equation}

Let us now check that $u$ is a weak solution. Since $\phii$ is continuous, and since $0 \le u_\nu \le 1$ a.e., $\phii(u_\nu)$ converges in $L^2(\OiT)$ towards $\phii(u)$. The $L^2((0,T);H^1(\O_i))$ estimate \eqref{phii(u)} does not depend on $u_0$, thus, up to a subsequence, $\left(\phii(u_\nu)\right)_\nu$ converges weakly to $\phii(u)$ in $L^2((0,T);H^1(\O_i))$. It also converges strongly in $L^2((0,T);H^s(\O_i))$ for all $s\in (0,1)$. This particularly ensures the strong convergence of the traces of $\left(\phii(u_\nu)\right)_\nu$ on the interface. Since $\phii^{-1}$ is continuous, we obtain the strong convergence of the traces of $\left( u_\nu \right)_\nu$. Checking  that the set 
$$
F= \{ (a,b)\in [0,1]^2 \ | \ \t\pi_1(a) \cap  \t \pi_2(b) \neq \emptyset \} \textrm{ is closed in } [0,1]^2,
$$
the limits $u_i$ fulfill $\t\pi_1(u_1) \cap  \t \pi_2(u_2) \neq \emptyset$, and so $u$ is a weak solution, then it is a SOLA.

If $u$ and $v$ are two SOLAs associated to the initial data $u_0$ and $v_0$, we can easily prove, using the Theorem \ref{unicite_bounded} that 
\begin{equation}\label{SOLA2}
\forall t\in [0,T],\qquad \sum_{i=1,2} \int_{\O_i} \phi_i \left( u (x,t) - v(x,t) \right)^\pm dx \le \sum_{i=1,2} \int_{\O_i} \phi_i \left( u_{0} (x) - v_{0}(x) \right)^\pm dx.
\end{equation}
The uniqueness particularly follows.

Let $u_0\in L^\infty(\O)$, $0 \le u_0 \le 1$, and let $\left(u_{0,\nu}\right)_\nu\subset \mathcal{E}$ a sequence of approximate initial data tending to $u_0$ in $L^1(\O)$. We denote by $u$ the unique SOLA associated to $u_0$, and by $\left(u_\nu\right)_\nu$ the bounded flux solutions associated to $\left(u_{0,\nu}\right)_\nu$.
Let $(M_p)_{p\in \N}$, $(N_p)_{p\in \N}$ be two sequences of positive integers tending to $+\infty$. Let $p\in \N$, $\nu\in\N$, 
let $u_{\Dd_p}$ the discrete solution corresponding to  $u_0$, and let 
$u_{\nu,\Dd_p}$ the discrete solution corresponding to  $u_{0,\nu}$. 
\begin{eqnarray*}
\| u_{\Dd_p} - u \|_{L^1(\OT)} &\le&  \| u_\Dd - u_{\nu,{\Dd_p}} \|_{L^1(\OT)} + \| u_{\nu, \Dd_p} - u_\nu \|_{L^1(\OT)} + \| u_\nu - u \|_{L^1(\OT)} . 
\end{eqnarray*}
From the discrete $L^1$-contraction principle  (\ref{comp_schema_2}), and from the continuous one \eqref{SOLA12}, we have 
\begin{eqnarray*}
\| u_{\Dd_p} - u \|_{L^1(\OT)} \le&& T  \frac{\max \phi_i}{\min\phi_i} \| u_{0,\Dd} - u_{0,\nu,{\Dd_p}} \|_{L^1(\O)} + \| u_{\nu, \Dd_p} - u_\nu \|_{L^1(\OT)}  \\
& +& T\frac{\max \phi_i}{\min\phi_i} \| u_{0,\nu} - u_0 \|_{L^1(\O)}.
\end{eqnarray*}
Letting $p$ tend to  $\infty,$ it follows from the definition of $(u_{0,\nu,\Dd})$ (adapted from \eqref{u0_d}) that 
$$
\lim_{p\to\infty} \| u_{0,\Dd} - u_{0,\nu,{\Dd_p}} \|_{L^1(\O)} = \| u_{0,\nu} - u_0 \|_{L^1(\O)}.
$$
We have proven in the Theorem \ref{thm_bounded} that the sequence of discrete solutions converges, under assumption on the initial data to the unique bounded flux solution, thus 
$$
\lim_{p\to\infty} \| u_{\nu, \Dd_p} - u_\nu \|_{L^1(\OT)} = 0.
$$
This implies 
$$
\limsup_{p\to\infty} \| u_{\Dd_p} - u \|_{L^1(\OT)} \le 2T\frac{\max \phi_i}{\min\phi_i} \| u_{0,\nu} - u_0 \|_{L^1(\O)}.
$$
Letting $\nu$ tend to $\infty$ provides 
$$
\lim_{p\to\infty} \| u_{\Dd_p} - u \|_{L^1(\OT)} = 0.
$$
The convergence occurs in $L^1(\OT)$, but the uniform bound on the sequence $\left( u_{\Dd_p} \right)$ in $L^\infty(\OT)$ ensures that the convergence also take place in all the $L^p(\OT)$, for $p\in [1,\infty)$.
\end{proof}
%%%%%%%%%%%%%%%%%%%%%%%%%%%%%%%%%%%%%%
%
% Numerical results
%
%%%%%%%%%%%%%%%%%%%%%%%%%%%%%%%%%%%%%%
\section{Numerical Result}\label{numer_section}

In order to illustrate this model, we use a test case developed by  Anthony Michel \cite{MCGP08}. The porous medium $\O=(0,1)$ is made of sand 
for $x\in (0,0.5) \cup (0.7,1)$, with a layer of shale for $x\in (0.5,0.7)$. 

{\bf First case:}\\
The total flow rate is equal to $0$, since $f_{\rm sand}(1) = f_{\rm shale}(1) = 0$, and the convection is the exclusive of the volume mass difference 
between the oil, which is lighter, and the water. 
The convection functions are given by:
$$
f_{\rm sand}(u) = 100 * f_{\rm shale}(u) = 50 * \frac{u^2 (1-u^2)}{1-2 u + 2 u^2}.
$$
The capillary pressures are first given by 
$$
\pi_{\rm sand}(u) = u^5, \qquad\qquad \pi_{\rm shale}(u) =0.5 + u^5.
$$
The function $\varphi_{\rm sand}$ and $\varphi_{\rm shale}$, given by 
$$
\varphi_{\rm sand}(u) = 10* \int_0^u  \frac{s^2 (1-s^2)}{1-2 s + 2 s^2} \pi_{\rm sand}'(s) ds, \qquad         
\varphi_{\rm shale}(u) = 0.1* \int_0^u  \frac{s^2 (1-s^2)}{1-2 s + 2 s^2} \pi_{\rm shale}'(s) ds,
$$
are computed using an approximate integration formula.
The initial data $u_0$ is equal to $0$, and $\uu = 0.001$, $\ou=0$.

The convection is approximated by a Godunov scheme, defined by 
$$
G_i(a,b)=\left\{\begin{array}{ll}
\ds \min_{s\in[a,b]} f_i(s) & \textrm{if } a\le b,\\
\ds \max_{s\in[b,a]} f_i(s) & \textrm{otherwise.}
\end{array} \right.
$$
%The implicit scheme is solved by a Newton method, 
\begin{figure} 
\begin{center} \includegraphics[height=0.2\hsize]{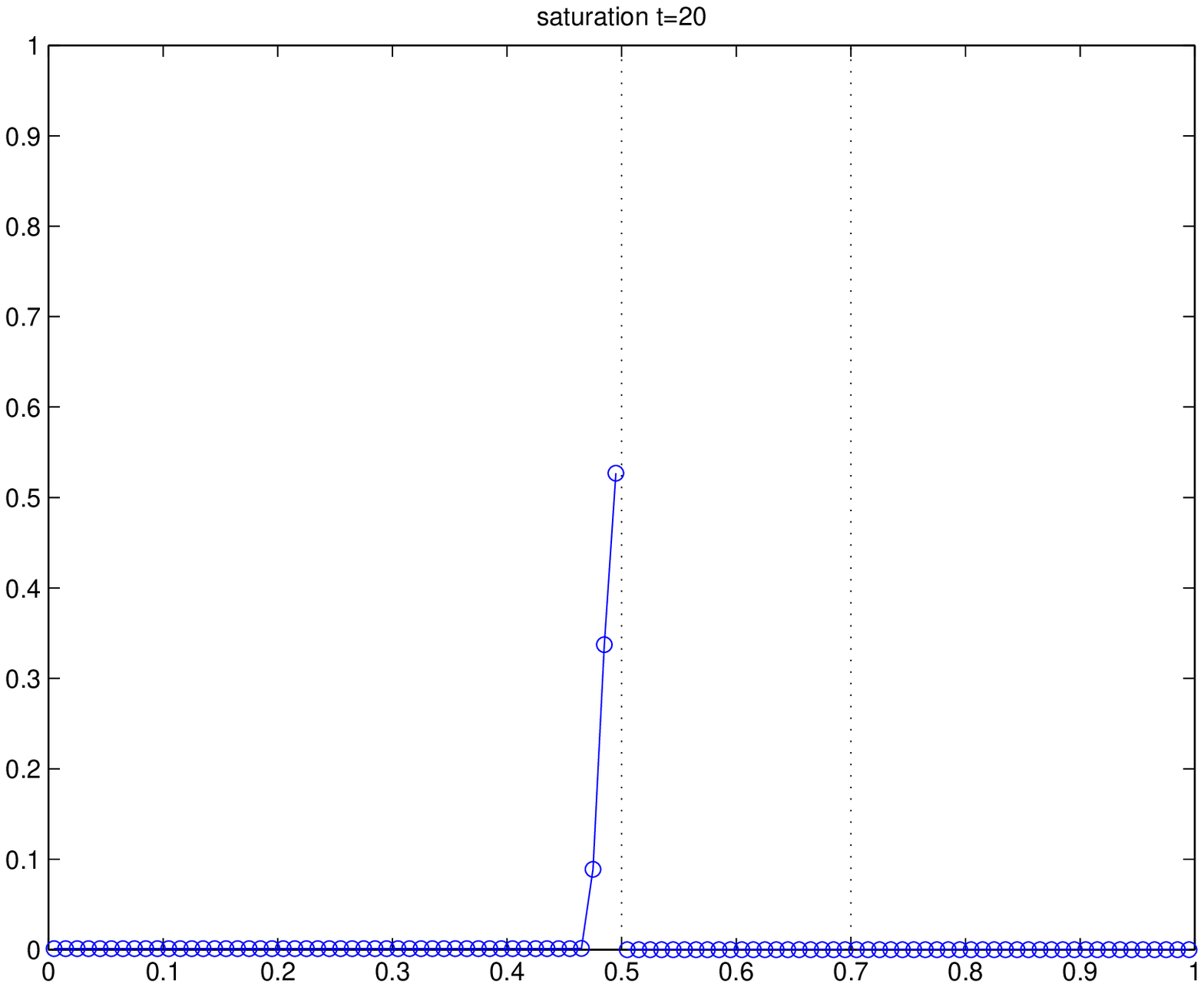} 
\quad \includegraphics[height=0.2\hsize]{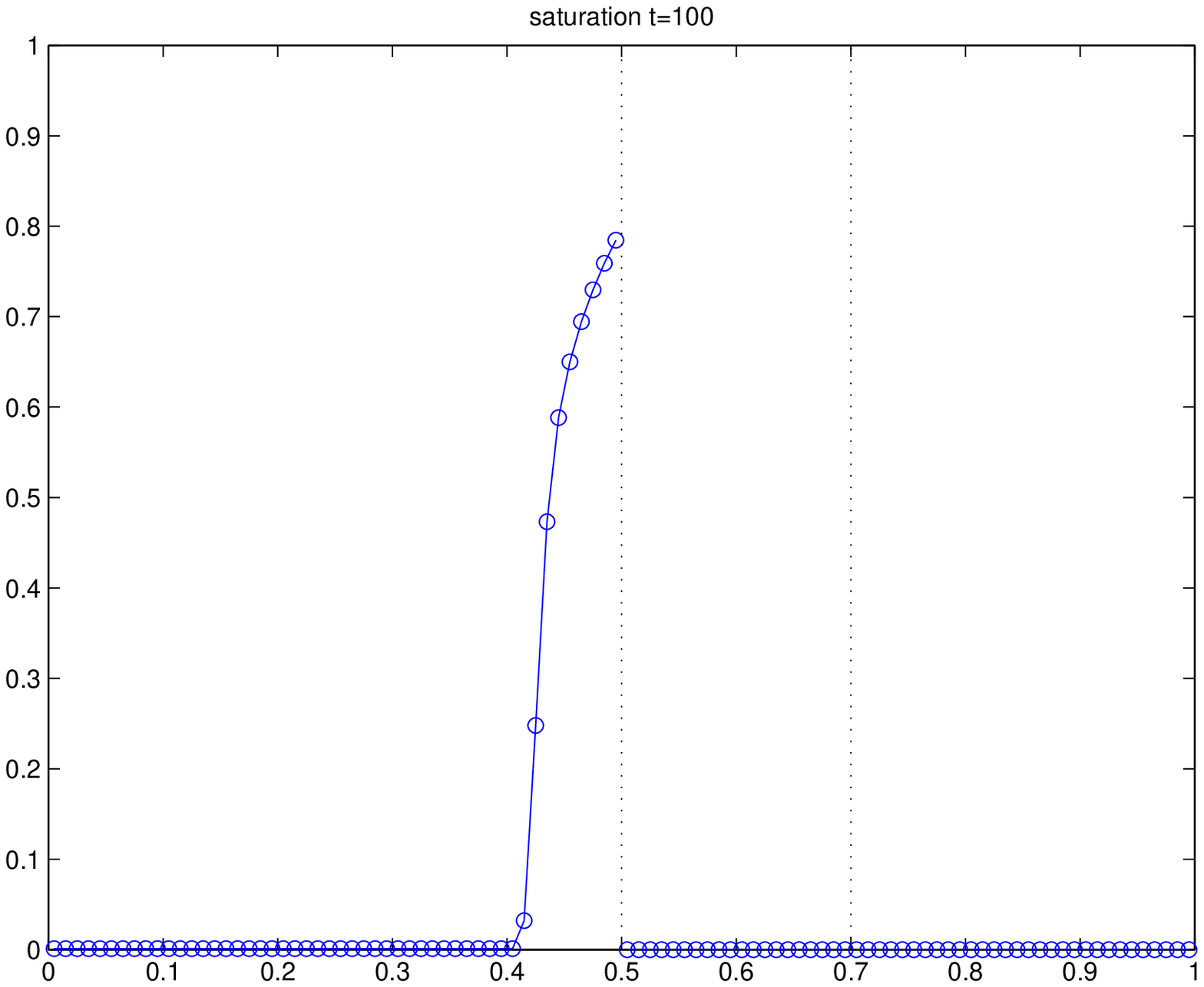} 
\quad\includegraphics[height=0.2\hsize]{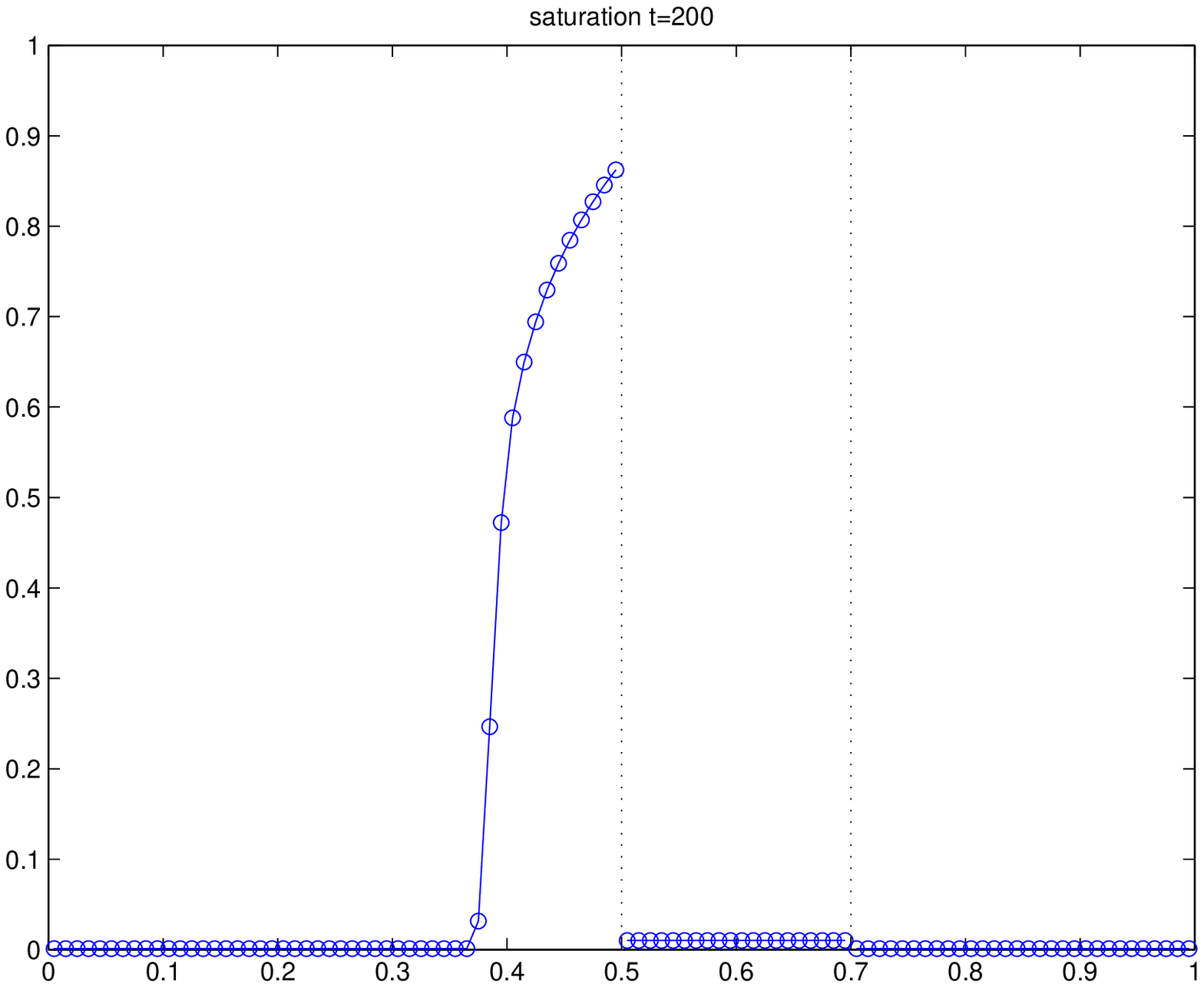} 
\caption{Saturation profiles for $t=20, t=100,t=200$}\label{sat_graph}
\end{center} 
\end{figure}

\begin{figure} 
\begin{center} \includegraphics[height=0.2\hsize]{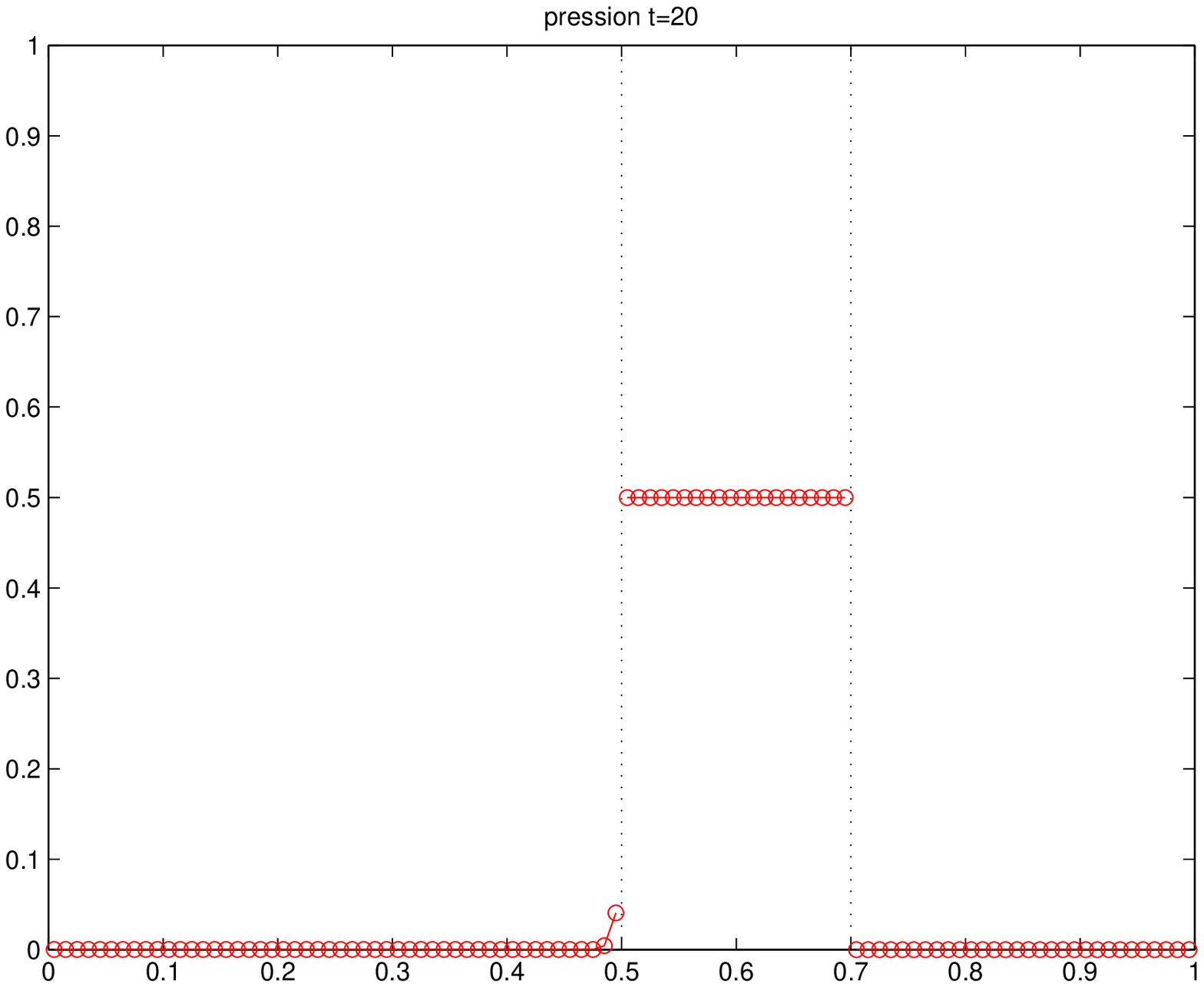} 
\quad \includegraphics[height=0.2\hsize]{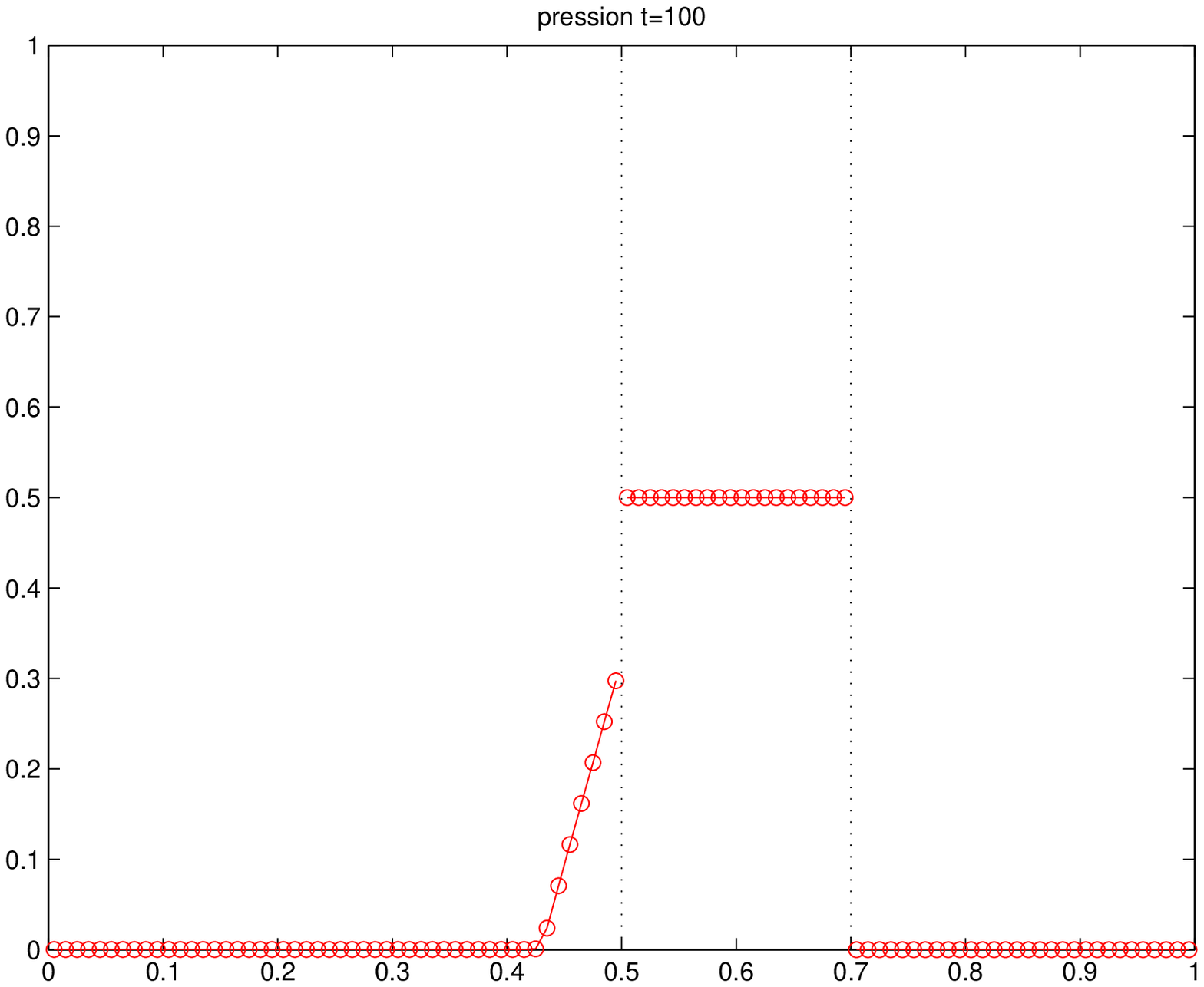} 
\quad\includegraphics[height=0.2\hsize]{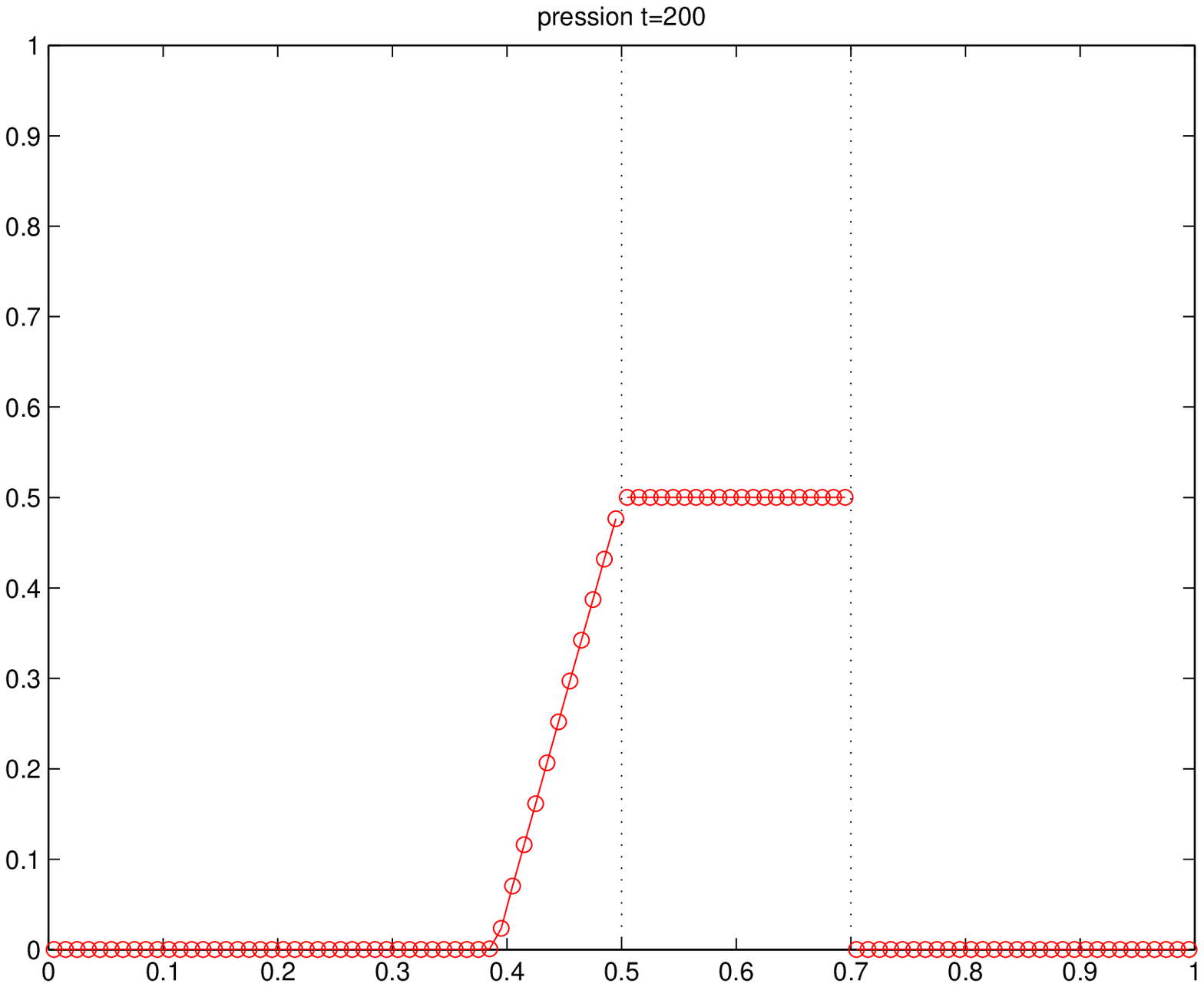} 
\caption{Capillary pressure profiles for $t=20, t=100,t=200$}\label{pres_graph}
\end{center} 
\end{figure}

\begin{figure} 
\begin{center} \includegraphics[height=0.2\hsize]{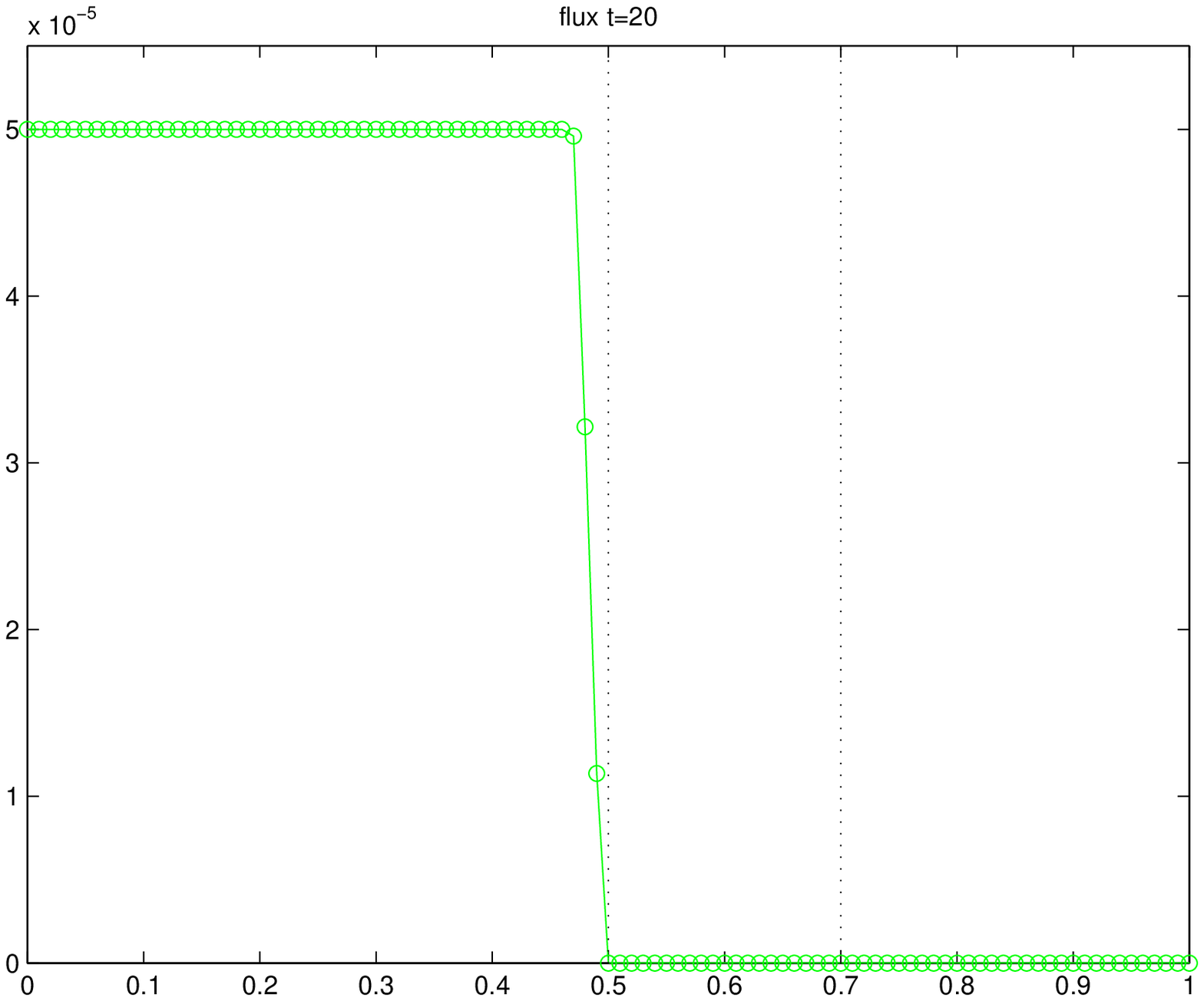} 
\quad \includegraphics[height=0.2\hsize]{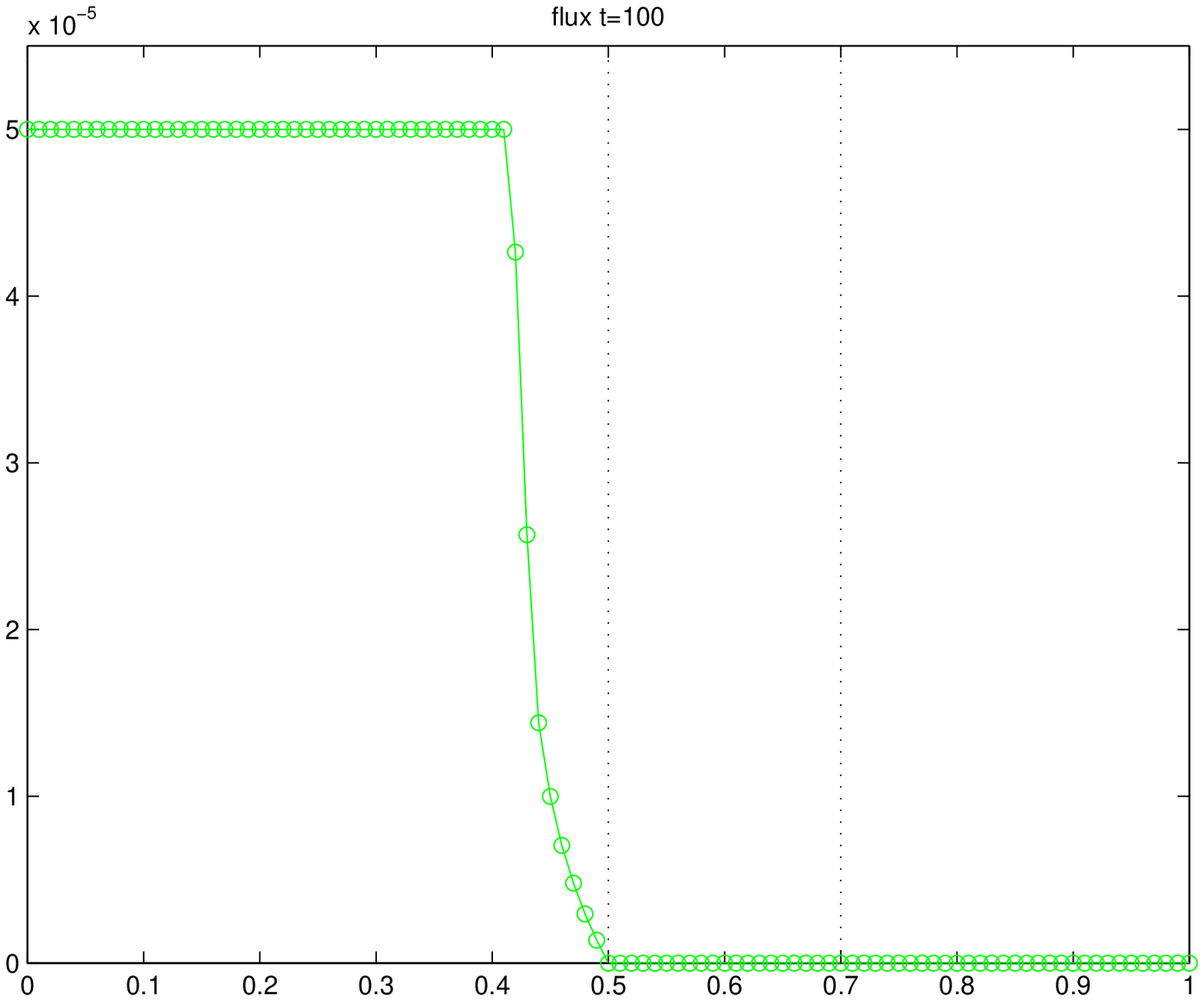} 
\quad\includegraphics[height=0.2\hsize]{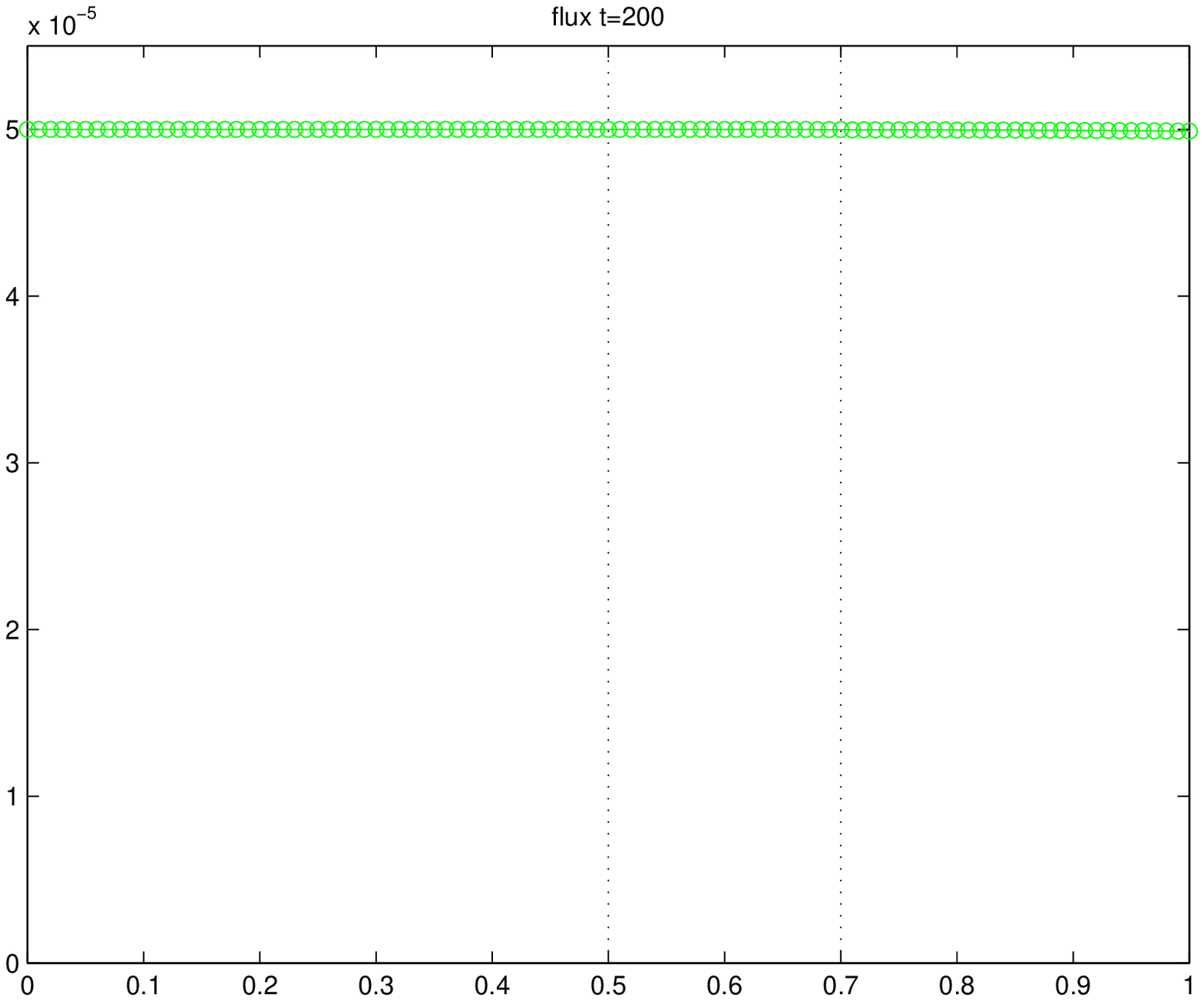} 
\caption{Oil-flux profiles for $t=20, t=100,t=200$}\label{flux_graph}
\end{center} 
\end{figure}

A little quantity of oil enters the domain from the left boundary condition, and it moves forward in the first part made of sand. 
The discontinuity of the capillary pressure (figure \ref{pres_graph}) stops the migration of oil, which begins to collect at the left of the 
interface, as shown on the figure \ref{sat_graph}.  One can check on the figure \ref{flux_graph} that for $t$ small enough, the oil-flux through the interface $\{x=0.5\}$ is equal to $0$. The accumulation of oil at the left of $\{x=0.5\}$ implies an increase of the capillary pressure. As soon as the 
capillary pressure connects at $\{x=0.5\}$, the oil can flow through the shale. The next discontinuity at $\{x=0.7\}$ does not impede the progression of the oil, since the capillary pressure force, oriented from the large pressure to the small pressure (here from the left to the right), works in the same direction that  the buoyancy, which drives the migration of oil. 

For $t=200$, the presented solution is a steady solution, with constant flux (figure \ref{flux_graph}). Some oil remains blocked in the first subdomain 
$(0,0.5)$. Even if one puts $\uu(t)=0$ for $t\ge 200$, the main proportion of oil in the porous medium can not overpass the interface $\{x=0.5\}$ and leave the porous medium $(0,1)$. Indeed, the function defined by 
$$
u^{\rm s}(x) = \left\{  \begin{array}{ll}
0 & \textrm{ if } x \notin (0.4,0.5),\\
\pi_{\rm sand}^{-1}(5(x-0.4)) & \textrm{ if } x \in (0.4,0.5)
\end{array}  \right.
$$
is a steady solution to the problem for $\uu=0$. It is easy to check that $u(\cdot,200) \ge u^{\rm s}$, thus the comparison principle stated in the Theorem \ref{unicite_bounded} ensures that for all $t\ge200$, $u(\cdot,t) \ge u^{\rm s}$. Thus for all $t\ge 200$
$$
\int_0^{0.5} u(x,t) dx \ge \int_0^{0.5} u^{\rm s}(x) dx >0.
$$
This quantity is said to be trapped by the geology change. Further illustrations, and a scheme comparison will be given in \cite{MCGP08}.

{\bf Second case:}\\
We only change the values of the capillary pressure functions (and also the linked functions $\varphi_{\rm sand}$ and $\varphi_{\rm shale}$).
The amplitude of the variation of each function is reduced from $1$ to $0.2$, i.e.
$$
\pi_{\rm sand}(u) = 0.2* u^5, \qquad\qquad \pi_{\rm shale}(u) =0.5 + 0.2*u^5.
$$
The graph transmission condition for the capillary pressure turns to 
$$
(1 - u_{\rm sand}) u_{\rm shale} = 0,
$$
where $u_{\rm sand}$ (resp $u_{\rm shale}$) denotes the trace of the oil saturation at the interfaces $\{x=0.5\}$ and $\{x=0.7\}$.
In this case, no oil can overpass the first interface, which is thus impermeable for oil. The only steady solution is 
$$
u^{\rm s}(x) = \left\{ \begin{array}{ll}
1 & \textrm{if } x<0.5,\\
0&\textrm{if } x>0.5.
\end{array}\right.
$$
An asymptotic study for capillary pressures tending to functions depending only of space, and not on the saturation has been performed in 
\cite[Chapter 5\&6]{These} (see also \cite{NPCX,NC_choc}). It has been proven that either the limit solution for the saturation is an entropy solution 
for th          e hyperbolic scalar conservation law with discontinuous fluxes in the sense of \cite{Tow00, Tow01,SV03, AV03, AJV04, AMV05, AMV07, Bac04,BV05,Bac06,Bachmann_these,Jim07} (see also \cite{KRT02a,KRT02b,KRT03}), mainly when the capillary forces at the interface are oriented in the same direction that the gravity forces, or that non-classical shocks can occur at the interfaces when the capillary forces and the gravity are oriented in opposite directions.
\begin{figure} [htb]
\begin{center} \includegraphics[height=0.2\hsize]{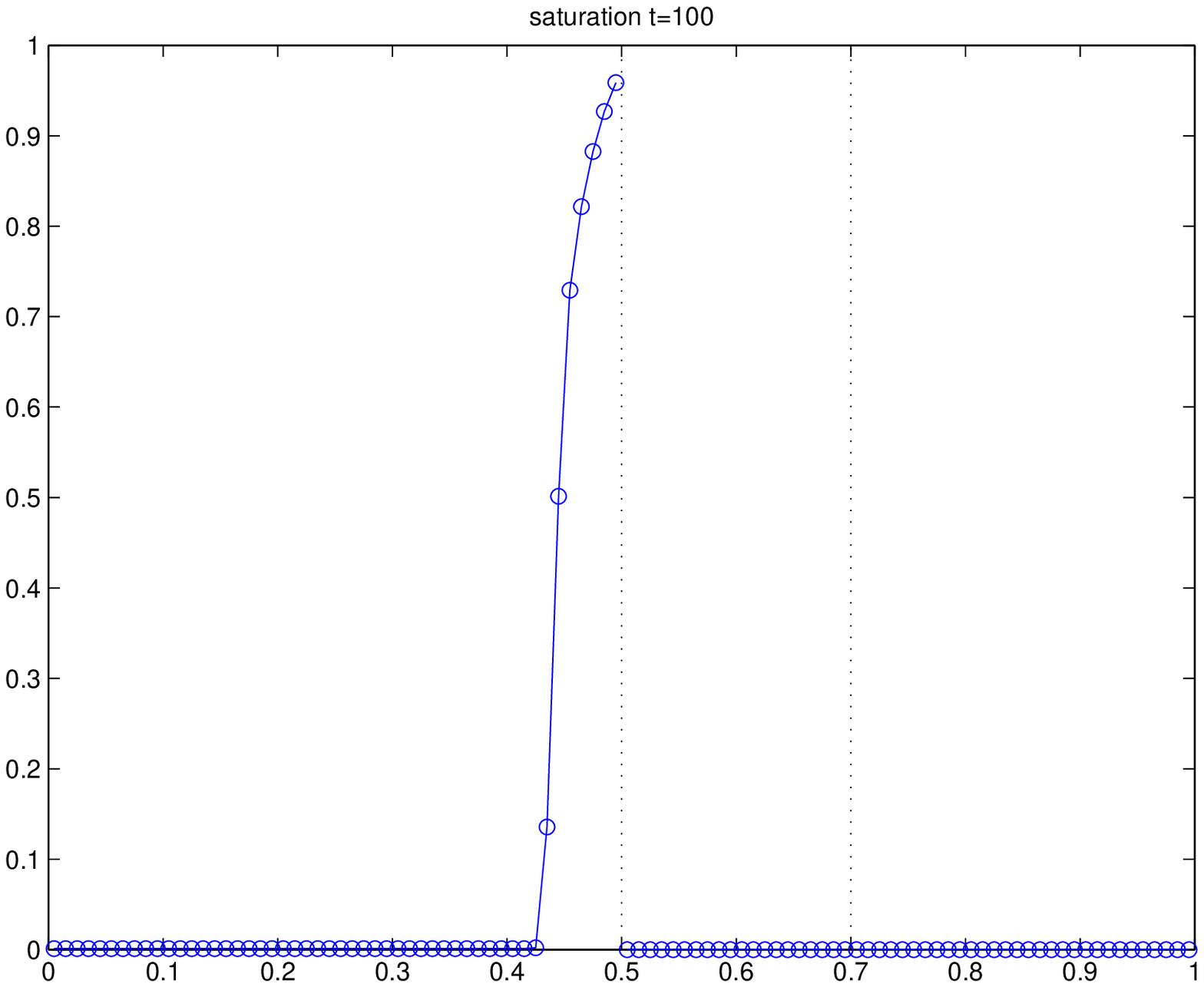} 
\quad \includegraphics[height=0.2\hsize]{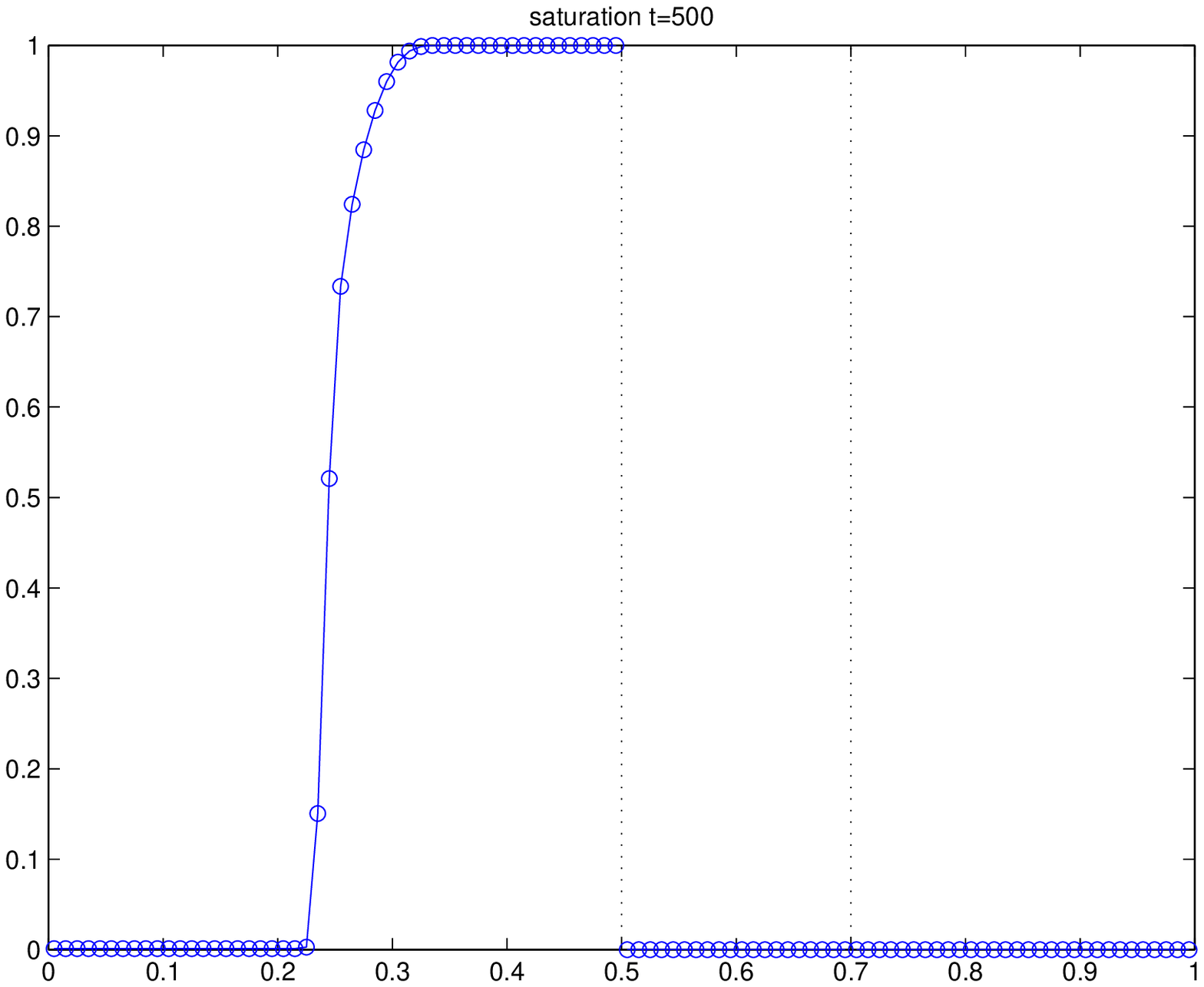} 
\quad\includegraphics[height=0.2\hsize]{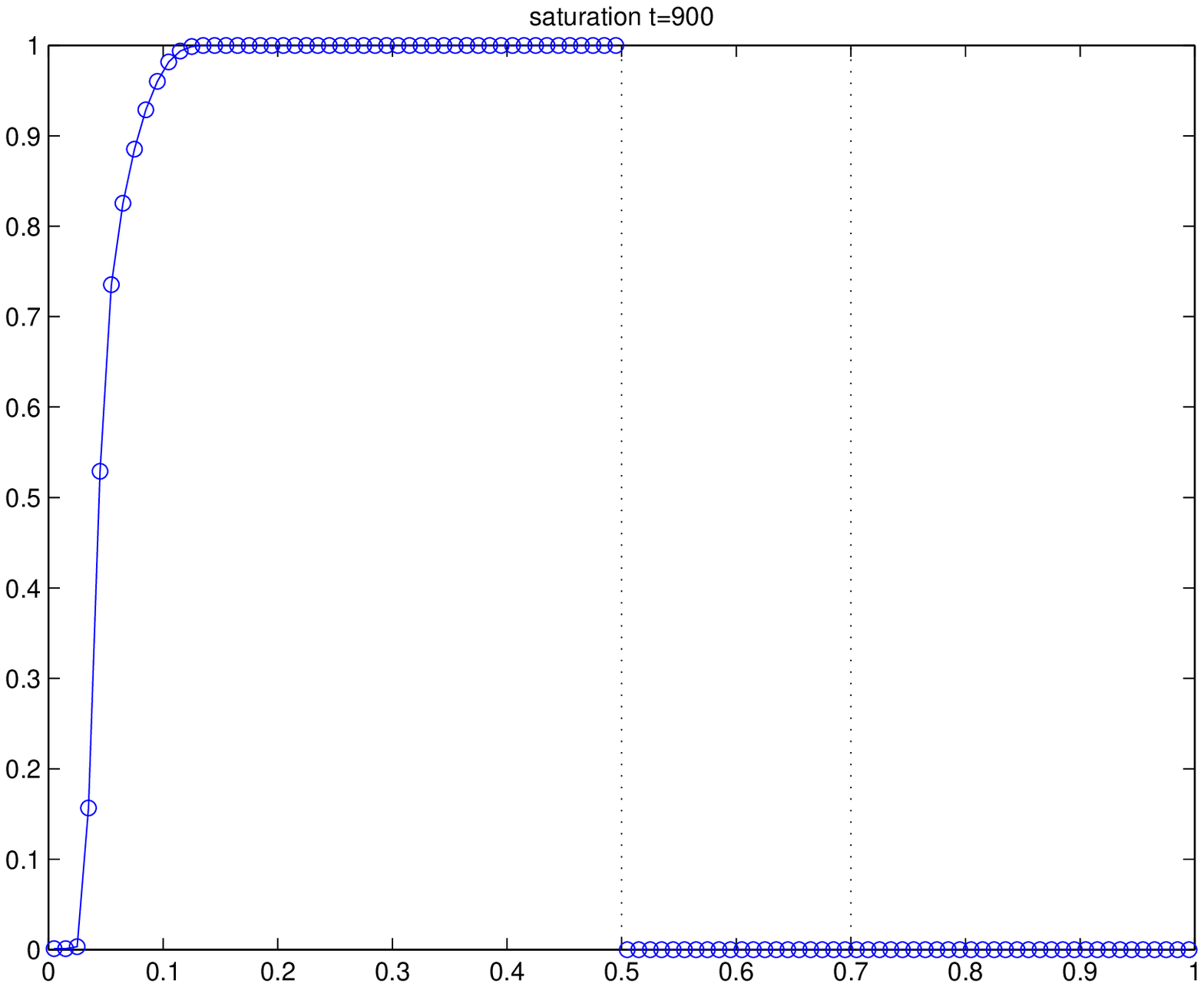} 
\caption{Saturation profiles for $t=100, t=500,t=900$}\label{sat_graph2}
\end{center} 
\end{figure}

\begin{figure} [htb]
\begin{center} \includegraphics[height=0.2\hsize]{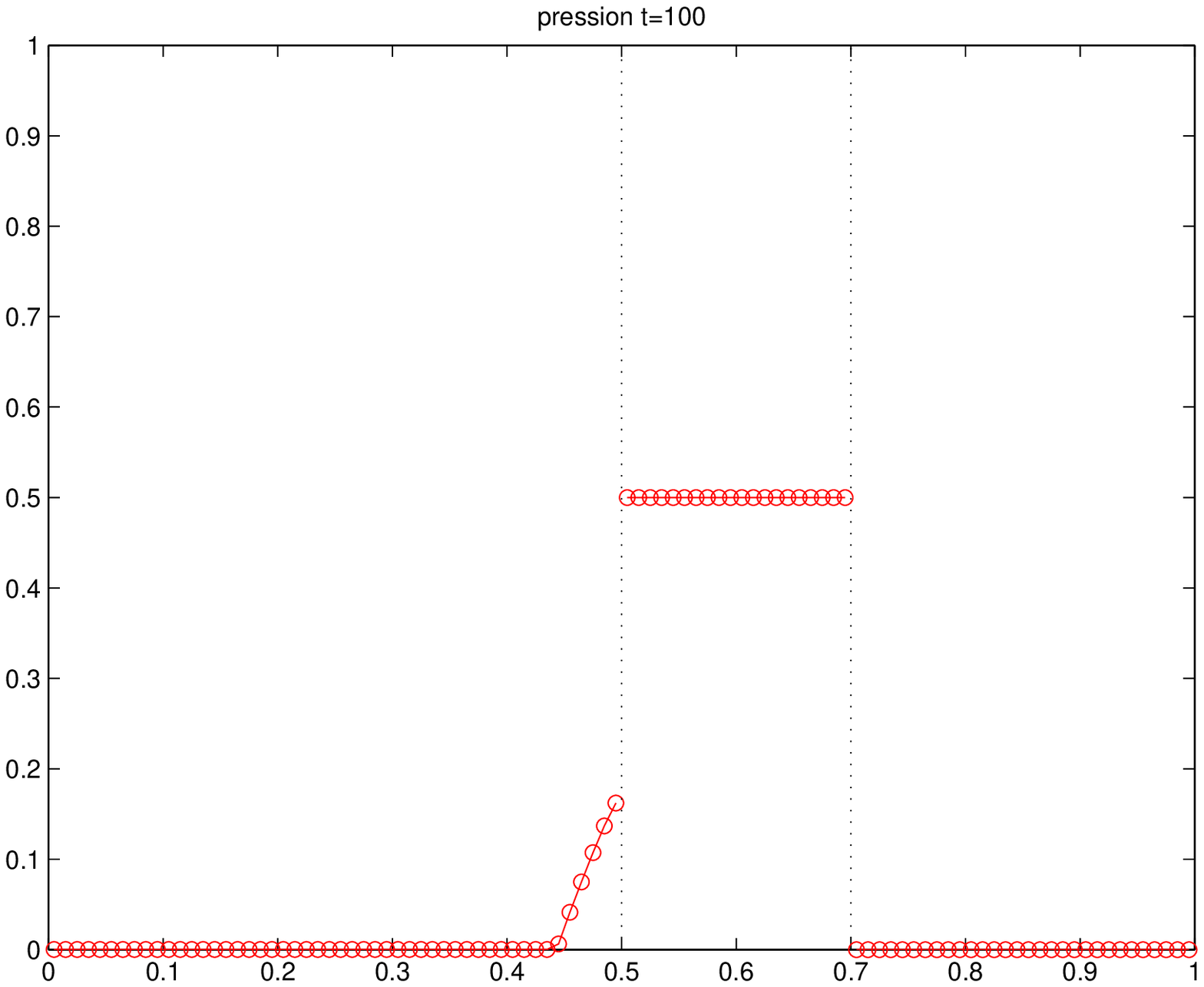} 
\quad \includegraphics[height=0.2\hsize]{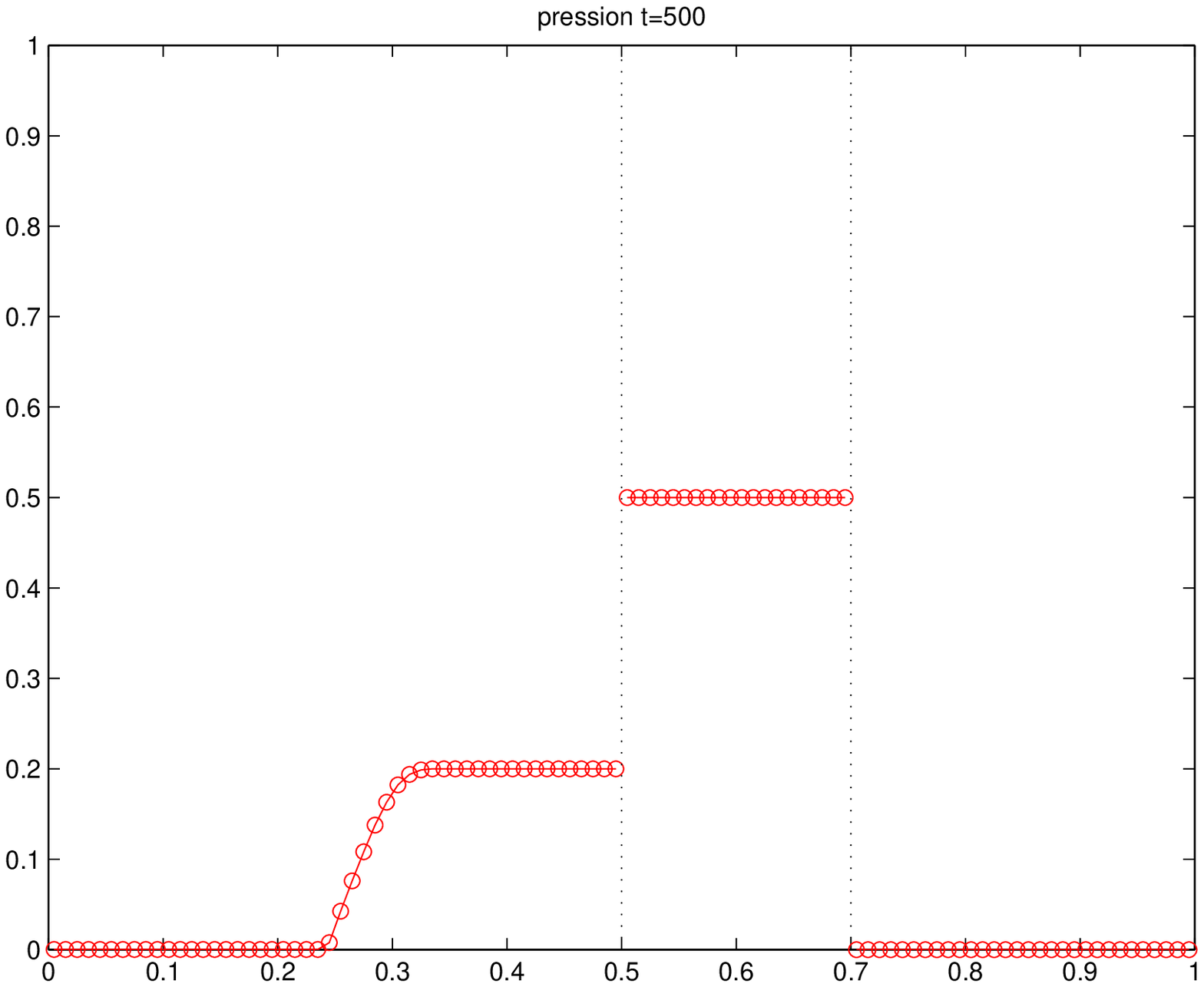} 
\quad\includegraphics[height=0.2\hsize]{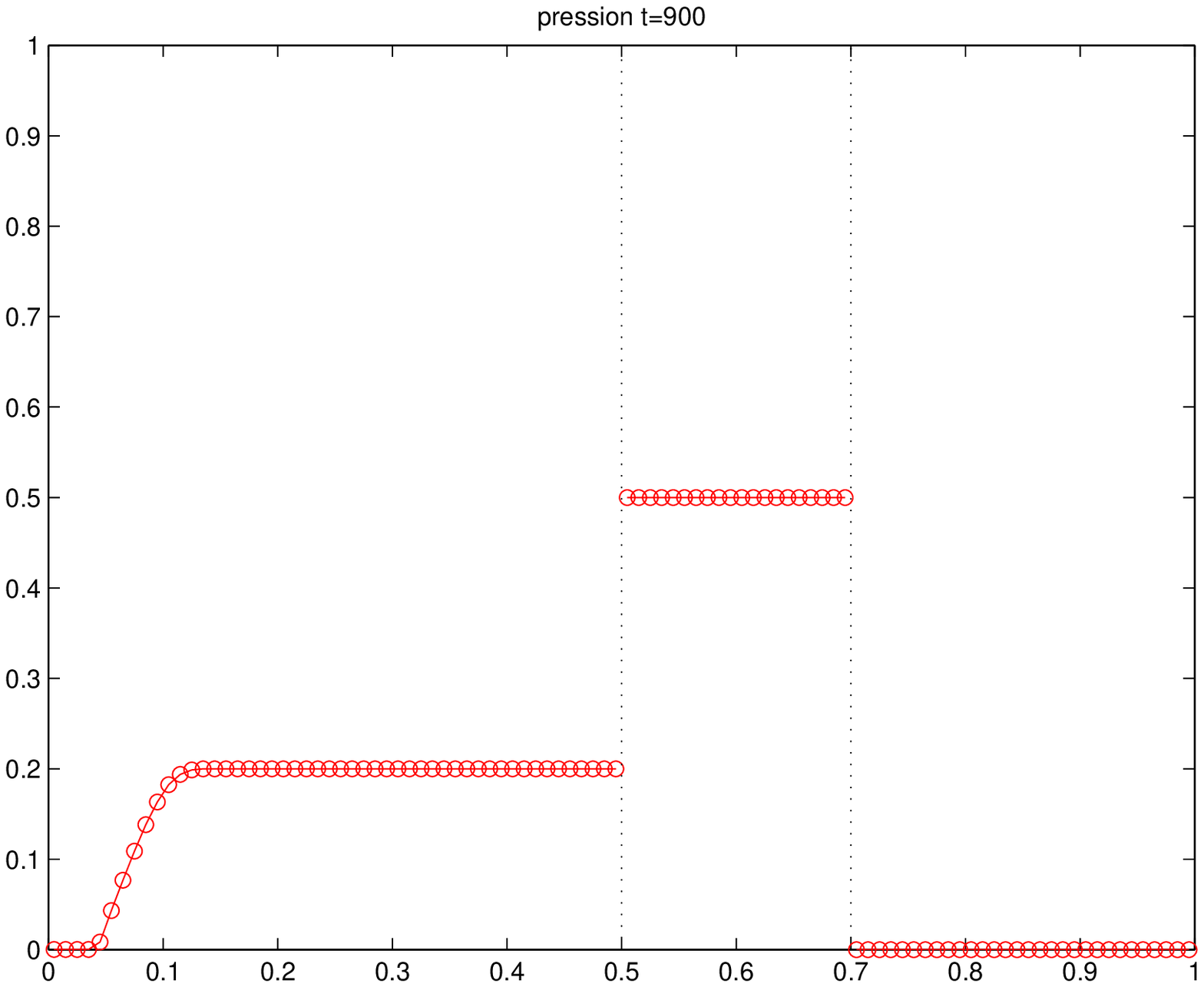} 
\caption{Capillary pressure profiles for $t=100, t=500,t=900$}\label{pres_graph2}
\end{center} 
\end{figure}

\begin{figure} [htb]
\begin{center} \includegraphics[height=0.2\hsize]{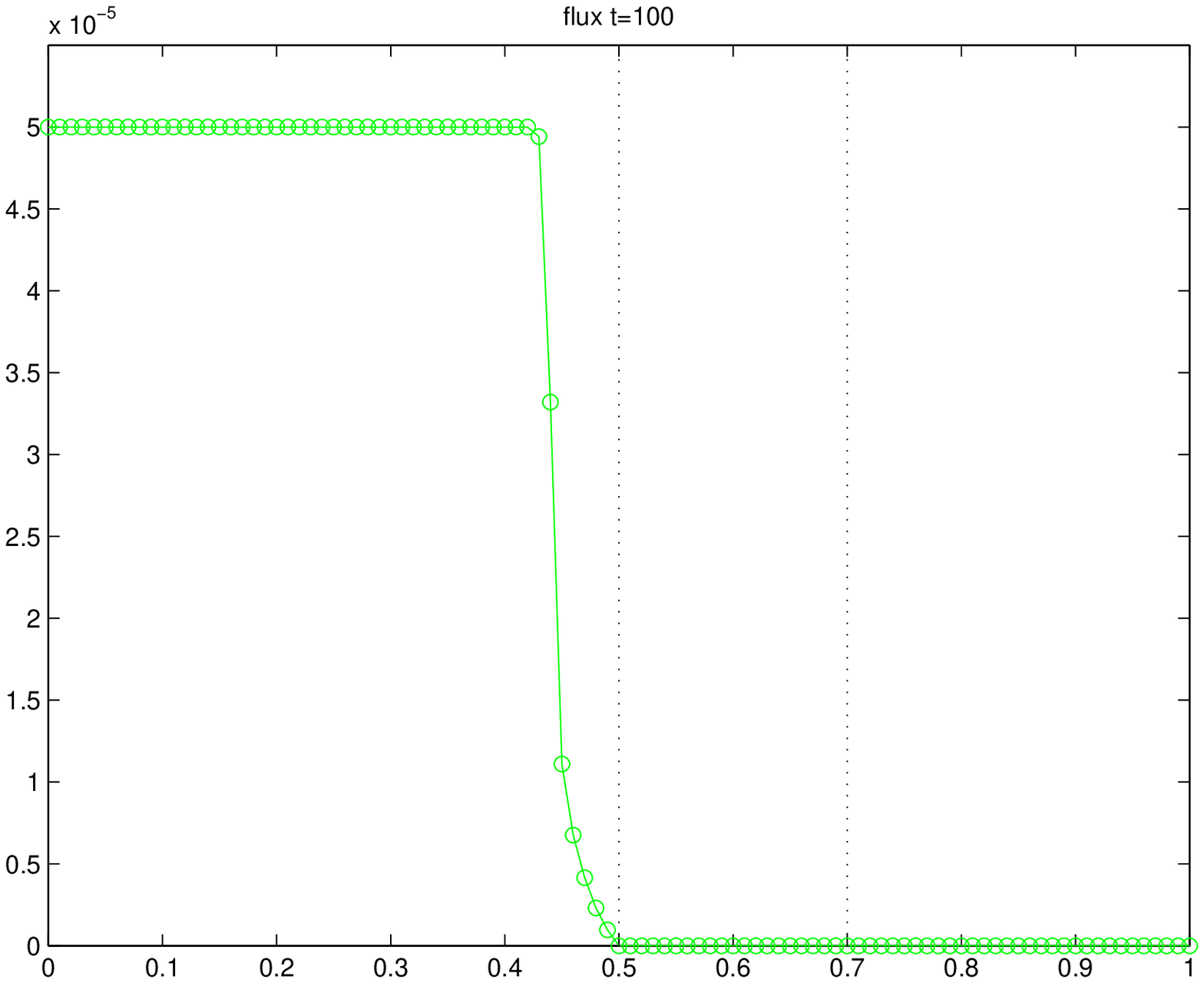} 
\quad \includegraphics[height=0.2\hsize]{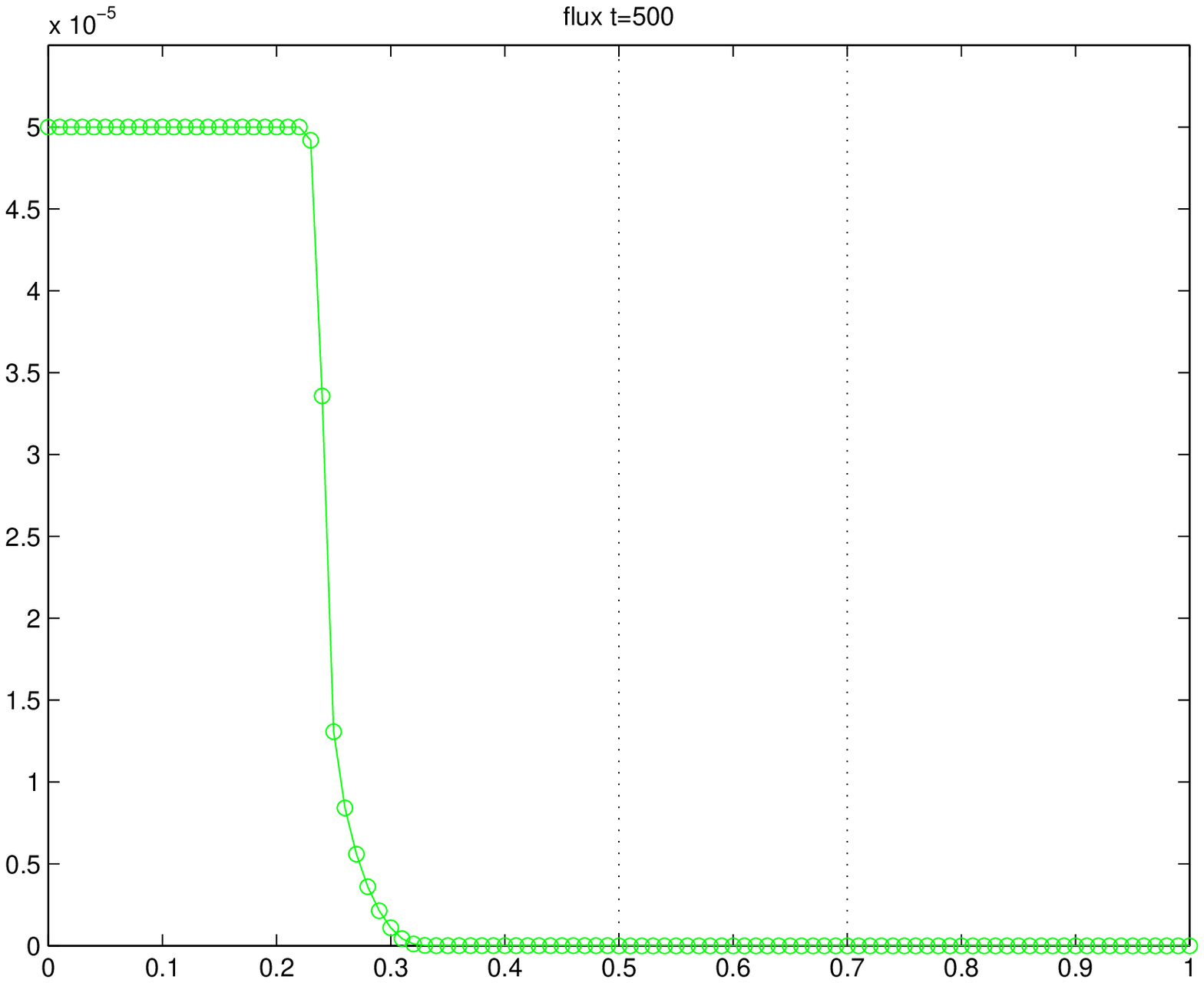} 
\quad\includegraphics[height=0.2\hsize]{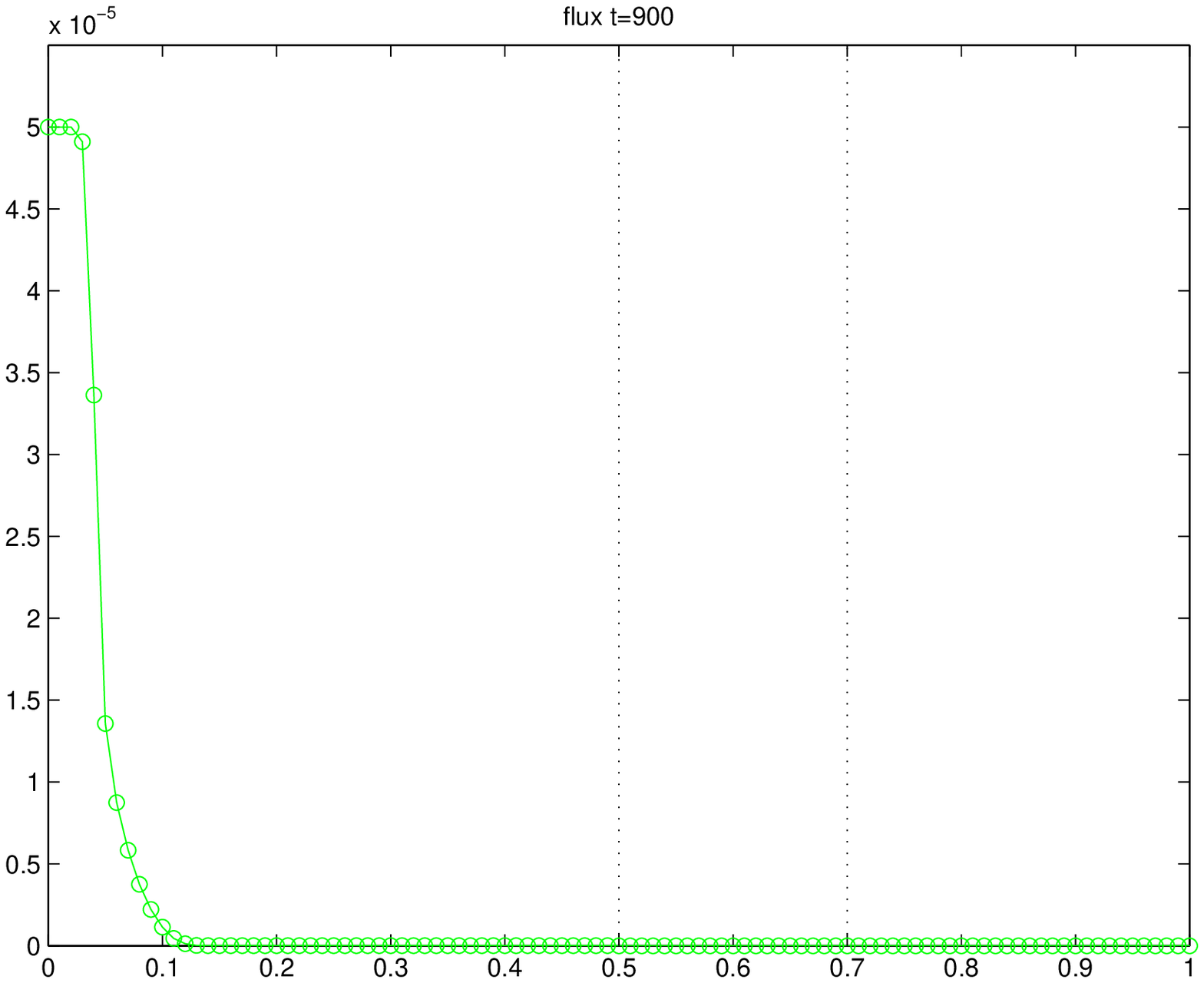} 
\caption{Oil-flux profiles for $t=100, t=500,t=900$}\label{flux2_graph}
\end{center} 
\end{figure}

\paragraph{Acknowledgements.}
The author would like to acknowledge the Professor Thierry Gallou\"et for his numerous recommendations and  Anthony Michel from IFP for the fruitful discussions on the models. He also thanks Alice Pivan for her help with the English language.

 %%-----------------------------
%%      your bibliography
%%-----------------------------
\bibliographystyle{plain}
\bibliography{ccances}
\end{document}